\documentclass[aos]{imsart}

\RequirePackage{amsthm,amsmath,amsfonts,amssymb}
\RequirePackage[numbers]{natbib}
\RequirePackage[colorlinks,citecolor=blue,urlcolor=blue]{hyperref}%% uncomment this for coloring bibliography citations and linked URLs
\RequirePackage{graphicx}%% uncomment this for including figures

\startlocaldefs

\usepackage{sven-draft}

\usepackage{mdframed}

\endlocaldefs

\begin{document}

\begin{frontmatter}
%\title{On global polynomial-time estimation in non-linear statistical inverse problems}
\title{Global polynomial-time estimation in statistical nonlinear inverse problems via generalized stability}
%\title{A sample article title with some additional note\thanksref{T1}}
\runtitle{Polynomial-time non-linear estimators}
%\thankstext{T1}{A sample of additional note to the title.}

%%%%%%%%%%%%%%%%%%%%%%%%%%%%%%%%%%%%%%%%%%%%%%%
%% Only one address is permitted per author. %%
%% Only division, organization and e-mail is %%
%% included in the address.                  %%
%% Additional information (such as           %%
%% indicating the corresponding author) can  %%
%% be included in the Acknowledgments        %%
%% section if necessary.                     %%
\begin{aug}
\author{\fnms{Sven}~\snm{Wang}}
\address{\'Ecole Polytechnique F\'ed\'erale de Lausanne (EPFL)\\
Institute of Mathematics\\
sven.wang@epfl.ch}
\end{aug}

\begin{abstract}
Non-linear statistical inverse problems pose major challenges both for statistical analysis and computation.  Likelihood-based estimators typically lead to non-convex and possibly multimodal optimization landscapes, and Markov chain Monte Carlo (MCMC) methods may mix exponentially slowly. We propose a class of computationally tractable estimators---plug-in and PDE-penalized M-estimators---for inverse problems defined through operator equations of the form $L_f u = g$, where $f$ is the unknown parameter and $u$ is the observed solution. 
The key idea is to replace the exact PDE constraint by a weakly enforced relaxation, yielding conditionally convex and, in many PDE examples, nested quadratic optimization problems that avoid evaluating the forward map $G(f)$ and do not require PDE solvers. For prototypical non-linear inverse problems arising from elliptic PDEs, including the Darcy flow model $L_f u = \nabla\!\cdot(f\nabla u)$ and a steady-state Schr\"odinger model, we prove that these estimators attain the best currently known statistical convergence rates while being globally computable in polynomial time. In the Darcy model, we obtain an explicit sub-quadratic $o(N^2)$ arithmetic runtime bound for estimating $f$ from $N$ noisy samples.
Our analysis is based on new generalized stability estimates, extending classical stability beyond the range of the forward operator, combined with tools from nonparametric M-estimation.
We also derive adaptive rates for the Darcy problem, providing a blueprint for designing provably polynomial-time statistical algorithms for a broad class of non-linear inverse problems. Our estimators also provide principled warm-start initializations for polynomial-time Bayesian computation.
\end{abstract}



\end{frontmatter}
%%%%%%%%%%%%%%%%%%%%%%%%%%%%%%%%%%%%%%%%%%%%%%
%% Please use \tableofcontents for articles %%
%% with 50 pages and more                   %%
%%%%%%%%%%%%%%%%%%%%%%%%%%%%%%%%%%%%%%%%%%%%%%
\tableofcontents

%%%%%%%%%%%%%%%%%%%%%%%%%%%%%%%%%%%%%%%%%%%%%%
%%%% Main text entry area:

\section{Main ideas and overview}\label{sec:intro}

Non-linear inverse problems have recently drawn significant research interest in modern statistics, applied and pure mathematics, and scientific computing \cite{S10,nickl2023bayesian,paternain2023geometric,BB18,kaipio2005statistical,I06}.
%Non-linear inverse problems can be notoriously challenging to solve both statistically and computationally.
They arise whenever an unknown statistical parameter $f\in \mathcal F$ can only be observed \textit{indirectly} through a nonlinear forward map
\[
\mathcal G: \mathcal F \to \mathcal H,
\]
where \(\mathcal F\) denotes a high- or infinite-dimensional parameter space, and \(\mathcal H\) is the space where observations can be obtained, subject to measurement noise. Typical examples arise in medical imaging, geophysical fluid dynamics, and weather and climate modelling, where \(f\) encodes an unknown parameter of the physical system. In such cases, the observation space is a Hilbert space \(\mathcal H\subseteq L^2(\mathcal O)\) consisting of functions on a spatial or spatio-temporal domain \(\mathcal O\subseteq \mathbb R^d\). In this paper, we consider statistical non-linear inverse problems where the data consist of noisy measurements
\begin{equation}\label{eq:obs}
Y_i = \mathcal G(f_0)(X_i) + \varepsilon_i,\qquad i=1,\dots,N,
\end{equation}
where \(f_0\in\mathcal F\) is some unknown ground truth data-generating parameter, \(X_i\in\mathcal O\) are design points, $\varepsilon_i\in \R$ are independent (e.g.~Gaussian) noise variables, and $N\ge 1$ denotes statistical sample size.

In the key class of inverse problems studied in this paper, \(\mathcal G\) is described by the non-linear parameter-to-solution map of a partial differential equation (PDE) which describes the observed system. The parameter $f$ plays the role of a coefficient function within the PDE, and $\mathcal G(f)\in L^2(\mathcal O)$ denotes the observable solution \cite{S10,nickl2023bayesian,kaipio2005statistical,I06}---we will introduce several concrete examples below. The main inferential task is to recover \(f_0\) from the data
\[Z^{(N)}=(Y_i,X_i)_{i=1}^N\] given by (\ref{eq:obs}). A secondary goal is given by the estimation of \(\mathcal G(f_0)\), which constitutes a \emph{PDE-constrained regression problem} over the (typically nonlinear) manifold of functions defined by the range of $\mathcal G$, $\big\{ \mathcal G(f): f\in\mathcal F \big\} \subseteq L^2(\mathcal O)$. The relevance of the preceding framework has led to significant methodological advances in applied mathematics and statistics---notably, this includes widely used maximum-likelihood and regularization-based estimators \cite{HNS95,KNS08} as well as Bayesian inversion techniques  \cite{S10, nickl2023bayesian}.

Devising statistical reconstruction methods which satisfy rigorous algorithmic guarantees has proven to be challenging. In practice, maximum likelihood and Bayes methods are reliant on the minimization of a (penalized) non-convex log-likelihood functional---for Gaussian noise equal to
\begin{equation}\label{likelihood}
    -\ell_N(f) = \frac 1{2N} \sum_{i=1}^N [Y_i-\mathcal G(f)(X_i) ]^2,~~~~~f\in\mathcal F,
\end{equation}
or on sampling from the high-dimensional non-log-concave Bayesian posterior distribution $d\Pi(f|Z^{(N)}) \propto \exp(\ell_N(f))d\Pi(f)$ using MCMC algorithms \cite{CRSW13,robert1999monte,MT09}. The difficulty in these tasks lies in the non-convexity of the negative log-likelihood \eqref{likelihood}, which is in turn due to the non-linearity of $\mathcal G$. This has rendered classical theory of convex optimisation \cite{BV04} and log-concave sampling \cite{D17, DM18} inapplicable, whence the computation of the above estimators could a priori be exponentially or `NP'-hard \cite{BMNW23}. While there has been a significant progress in establishing statistical consistency results and convergence rates \cite{NVW18, MNP21,AW24,GN19} as well as Bernstein-von Mises type theorems \cite{N20,MNP21-AOS}, the literature on \textit{algorithmic} or \textit{computational} guarantees---concerned with the question of how to numerically access statistical estimators---
has remained relatively limited \cite{NW22,BN24,nickl2023bayesian,HSV14} and relies on strong assumptions (as will be discussed further below).

%have been developed by combining analytical PDE tools with seminal ideas in nonparametric statistics \cite{V00,GGV00}; see the recent monograph \cite{nickl2023bayesian} for an overview.

The central guiding problem of this paper is to develop methods which both possess provable \emph{statistical guarantees} and \emph{global polynomial-time computability}, and can be summarized as follows. The notion of computational cost considered here is based on the number of elementary arithmetic operations (such as additions and multiplications of real numbers), and will be formalized below.

%Let us formulate this question concretely:

\begin{mdframed}
\textit{Existence problem for polynomial-time estimation methods.} Does there exist a sequence of estimators, i.e. measurable functions of the data $\hat f_N= \hat f_N(Z^{(N)})$, which
\begin{itemize}
\item[(i)] have a polynomial statistical convergence rate,
  \[ \forall N\ge 1:~~ \mathbb E_{f_0}\|\hat f_N-f_0\|^2 \le CN^{-\gamma}~~~\text{for some}~~~C>0,\gamma >0,~~~\text{and}
  \]
\item[(ii)] have a computational cost $c_N>0$ growing at most polynomially as $N\to\infty$.
\end{itemize}
\end{mdframed}

Our main contribution is to propose and develop theory for two statistical procedures which we respectively call `PDE-penalized M-estimators' and `plug-in M-estimators', which extend classical M-estimation, are computationally tractable and do not require access to PDE solvers. These methods are based on enforcing PDE constraints \textit{weakly}. As a key prototypical non-linear inverse problem, we consider the \textit{Darcy flow} model, which has served as a main testbed for new theory and methodology in applied mathematics and statistics \cite{S10,CRSW13,nickl2023bayesian,S10, nickl2023bayesian, AW24, GN19,NVW18, CDS11,FNO}. To the best of our knowledge, no provably polynomial-time statistical inversion method is currently available; thus our approach yields novel polynomial-time feasibility guarantees (see Theorem \ref{darcy-main} below) for the Darcy model, with explicit \emph{sub-quadratic} $o(N^2)$ arithmetic runtime guarantees. A second example is given by a non-linear inverse problem for the Schr\"odinger equation, e.g.~relevant in photoacoustic tomography \cite{BU10,N20,NW22}. 

A conceptual novelty of this paper is to introduce a new notion of \emph{generalized stability} which extend classical stability stability beyond the range of the forward operator. These new estimates are key in our proofs and yield a unified approach to proving both (i)  statistical convergence rates and (ii) global polynomial-time computability for the proposed generalized M-estimators. A further consequence is that our polynomial-time estimators provide principled warm-start initializations for high-dimensional Bayesian computation, thereby enabling polynomial-time algorithms for approximating posterior expectations and for posterior sampling in the models considered; see Section~\ref{sec-intro-warmstart}.

%\section{Main ideas and results}

\subsection{Generalized M-estimators} \label{sec-motivation}

Before giving an overview of the main results, let us explain the intuition underlying the proposed estimation procedures, which we shall refer to as \textit{generalized M-estimators}. The estimators can be motivated using two very simple ideas for modifying the non-convex (possibly intractable) likelihood landscape \eqref{likelihood}. These ideas consist of (i) the  \textit{augmentation of the parameter} and (ii) a certain \textit{weak formulation} of the PDE constraint defining the map $\mathcal G$. We will consider forward maps $\mathcal G$ which are defined implicitly by an operator equation
\begin{equation}\label{PDE}
L_f[u] = g~~~+~~~\text{boundary conditions},
\end{equation}
where $g$ is some known source term and $(L_f:f\in\mathcal F)$ represents a family of (differential) operators. Suppose that for each $f\in\mathcal F$, $\mathcal G(f)$ is given as the unique solution $u$ of \eqref{PDE}. This setting includes many prototypical parameter inference problems for PDEs---such as the Darcy flow problem where $L_fu= \nabla \cdot (f\nabla u)$ and the steady-state Schr\"odinger model where $L_fu= \Delta u/2 -fu$, as well as several other non-linear PDE examples such as the McKean-Vlasov and reaction-diffusion equations discussed further below.
%In general, we assume that $\mathcal F\subseteq \mathbb H$ is some (possibly non-linear) 
%For instance, in Darcy flow, where $\mathcal L_f[u]=\nabla\cdot(f\nabla u)$, if $g\in C^\infty(\mathcal)$ and $f$ possesses regularity $C^{1+\eta}(\mathcal O)$ and is strictly positive, the existence of a unique solution $u=\mathcal G(f)$ in the space $C^{2+\eta}(\mathcal O)$ is ensured by classical PDE theory \cite{GT98} (Chapter 6).
%In typical non-linear inverse problems of this type, both the parameter set $\mathcal F$ as well as the solution set $
%\{\mathcal G(f):f\in\mathcal F\}$ may be non-linear and non-convex subsets embedded in some ambient function space.

%\paragraph*{Augmentation of parameter}

In place of considering the likelihood landscape $f\mapsto \ell_N(f)$, we instead decouple the regression function $u$ (observed solution of the PDE) from the statistical parameter $f$, and to link the pair $(u,f)$ via an additional term in the objective.
%A first such approach would lead to the for some set $\mathcal U$ is some set of candidate regression functions, 
%augment the parameter space from $f\in\mathcal F$ to the pair $(f,u)\in \mathcal F\times \mathcal U$, where . Intuitively speaking, this step decouples the loss incurred by the regression $u \in \mathcal U$ from an additional loss which measures the discrepancy between $u$ and the physics `predicted by $f\in\mathcal F$'.
A first (naive) application of this principle would result in a minimization problem
\begin{equation}\label{relax-step1}
 \text{minimize}~~ J(f,u) = \frac 1N \sum_{i=1}^N [Y_i-u(X_i) ]^2
    + \lambda^2 \|u - \mathcal G(f)\|_{L^2(\mathcal O)}^2 + \mathcal R(u,f).
\end{equation}
Here and below, $\lambda >0$ is some regularisation parameter and $\mathcal R$ represents some (e.g.~Tikhonov-type) complexity penalty. The $L^2(\mathcal O)$-norm was chosen here in resemblance of the information-theoretic Kullback-Leibler divergence $KL(f_1,f_2)\approx \frac 12\|\mathcal G(f_1)-\mathcal G(f_2) \|_{L^2}^2$ (with equality holding for uniformly sampled $X_i$ and $\eps_i\sim N(0,1)$). The decoupling permits to further modify the `PDE penalty' $\|u - \mathcal G(f)\|_{L^2}^2$, which operates on the level of the regression functions, to $\tnorm{L_fu-g}$ for some suitable norm $\tnorm{\cdot}$.%, leading to the first method we will study.
\begin{mdframed}
    \textit{Method I: PDE-penalized M-estimator.} 
    \begin{equation}\label{relax-step2}
      \text{minimize}~~ J(f,u) =\frac 12 \sum_{i=1}^n (Y_i -u(X_i))^2 + \lambda^2 \tnorm{L_f u - g}^2 + \mathcal R(u,f).
\end{equation}
\end{mdframed}
The idea of minimizing a PDE penalty to circumvent potentially costly evaluations  of PDE solvers has appeared in key methodologies, notably in physics-informed neural networks (PINNs) \cite{PINN}, so-called `all-at-once' formulations of inverse problems \cite{kaltenbacher2016regularization}, optimal control \cite{troltzsch2010optimal} and PDE-constrained regression \cite{HPUU08} and `calming' techniques in nonparametric statistics \cite{S23}, see below for more discussion.
%For example, in the Darcy flow example one requires non-negativity constraint $\inf_{x\in\mathcal O}f(x) \ge f_{min}$ which ensure the uniform ellipticity of the differential operators $L_f[\cdot]$. This is typically enforced using a non-linear reparameterization of the parameter $f$; see for instance Section \red{XXX} below. In contrast, the objective \eqref{relax-step2} is well-defined for \textit{any $f$}, even when $L_f[\cdot]$ violates uniform ellipticity (and thus $L_f^{-1}$ may not be well-defined), allowing to formulate \eqref{relax-step2} as an unconstrained minimization problem.
%The problem \eqref{relax-step2} can thus be thought of as an unconstrained minimization over pairs $(f,u)\in \mathbb F\times \mathcal U$, where both $\mathbb F$ and $\mathcal U$ are linear spaces.

The second estimation procedure studied in this paper is given by yet another simplification of \eqref{relax-step2}, which we term \textit{plug-in M-estimation}. Here one further decomposes the minimization task into two nested problems:
\begin{mdframed}
\textit{Method II: Plug-in M-estimator.}
    \begin{equation}\label{plugin-1}
\begin{split}
        \hat u_N \in \argmin_{u\in \mathcal U}J_1(u),&~~~J_1(u) := \frac 12 \sum_{i=1}^n (Y_i -u(X_i))^2 + \mathcal R_1(u),\\
        \hat f_N \in \argmin_{f\in \mathcal F}J_2(f),&~~~J_2(f) := \lambda^2 \tnorm{L_f \hat u_N - g}^2 + \mathcal R_2(f).
\end{split}
\end{equation}
\end{mdframed}
One can, of course, think of the the plug-in M-estimator entirely independently of the above motivation. In fact, the `plug-in principle' of first performing a pure regression or `data smoothing' step then an inversion step has appeared in the literature, see the paper \cite{bissantz2007convergence} for linear inverse problems as well as the paper \cite{KSV24} discussed futher below. The key point is that the non-linear inversion itself is formulated as a computationally tractable \textit{minimization of the PDE-constraint}. In light of this, our motivation provided above gives a unified way of thinking about Methods I and II as \textit{generalized M-estimators}. We also note that the preliminary estimator $\hat u_N$ of $\mathcal G(f_0)$ need not necessarily arise from a penalized least squares method, but could arise from other nonparametric regression methods.

The formulations \eqref{relax-step2} and \eqref{plugin-1} are computationally attractive from several points of view. The first obvious reason is that the minimizations of \eqref{relax-step2}-\eqref{plugin-1} no longer require evaluations of $\mathcal G(f)$ (which likelihood-based methods require), circumventing possibly expensive access to PDE solvers. Secondly, in many key examples such as the Darcy and Schr\"odinger models, $(f,u)\mapsto L_fu$ is bi-linear, whence the objective \eqref{relax-step2} is \textit{conditionally quadratic}. The plug-in method \eqref{plugin-1} becomes a nested quadratic problem. The third reason relates to the range of functions $f$ over which optimization is permitted. Often the map $\mathcal G(f)$---which involves the \textit{inversion} of the differential operator $L_f$---is only well-defined on a non-linearly constrained subset $\mathcal F$ (e.g.,~non-negative functions). In contrast, the differential operator $L_fu$ is typically well-defined for \textit{arbitrary} sufficiently smooth functions $f$, which permits formulating \eqref{relax-step2} as an unconstrained minimization problem. This will be  convenient in the non-linear problems considered here, see Sections \ref{sec-intro-polytime} and \ref{sec:pde-results}.

%\todo{In fact, when $(f,u)\mapsto L_fu$ is bilinear, one may view this as \textit{a single iteration} of a coordinate-wise optimization scheme for \eqref{relax-step2}, first optimizing with respect to $u$ (conditional on $f=0$) and then with respect to $f$, conditional on $\hat u_N$.}

\subsection{A global $o(N^2)$-runtime theorem for the Darcy model}\label{sec-intro-polytime}

We illustrate the scope of our contributions in the Darcy model. For any sufficiently smooth and strictly positive \textit{conductivity function} $f:\mathcal O\to(f_{min},\infty)$, $f_{min}>0$, the forward map $\mathcal G(f)$ is defined as the solution $u=\mathcal G(f)$ to the boundary value problem
\begin{equation}\label{Darcy}
    \nabla\cdot(f\nabla u) = g \quad \text{on }\mathcal O, \qquad u|_{\partial \mathcal O}=0,
\end{equation}
where \(g\) denotes a fixed, known source term. The goal lies in identifying $f$ from noisy data of $u$. For convenience, we assume throughout that $g$ is infinitely differentiable. Then, for any $f\in C^{1+\eta}(\mathcal O)$ with $\eta>0$, there exists a unique classical solution $u\in C^{2+\eta}(\mathcal O)\cap C(\bar{\mathcal O})$, see Theorem 6.14 in \cite{GT98}. The Darcy flow problem has served as a prototypical inverse problem and testbed for MCMC methodology \cite{CRSW13} and theory \cite{nickl2023bayesian}, Bayesian inversion \cite{S10, nickl2023bayesian, AW24, GN19}, nonparametric M-estimation \cite{NVW18}, `parametric PDEs' \cite{CDS11} and scientific machine learning \cite{FNO, PINO, reinhardt2024statistical}, as well as classical theory for inverse problems with deterministic noise \cite{KNS08}. To the best of our knowledge, no polynomial-time statistically consistent estimator is currently known (both in the statistical and deterministic setting).

%This paper provides affirmative answers to the above question for several concrete PDE-based inverse problems and, more broadly, develops a general blueprint for analyzing nonlinear inverse problems.
The following theorem asserts that the inverse problem can be solved in \textit{sub-quadratic} $o(N^2)$ runtime while retaining the same statistical convergence rate as previously achieved by maximum likelihood and Bayes methods \cite{NVW18,GN19}. To formally state the theorem, let \(\mathcal O\subseteq \mathbb R^d (d\ge 1) \) be a smooth and bounded domain and for integer $\alpha\ge 1$, let $H^\alpha(\mathcal O)$ denote the usual $L^2$-Sobolev space of $\alpha$-times weakly differentiable functions on $\mathcal O$. We assume that the regression measurements $Z^{(N)}=(Y_i,X_i)_{i=1}^N$ arise from (\ref{eq:obs}) with 
\[ Y_i= \mathcal G(f_0)+\eps_i,~~\text{with}~~X_i\sim^{i.i.d.} \text{Unif}(\mathcal O)~~\text{and}~~\eps_i\sim^{i.i.d.} N(0,1),\]
where $\text{Unif}(\mathcal O)$ denotes the uniform distribution on $\mathcal O$.

\begin{thm}\label{darcy-main}
Fix integer $\alpha > (d/2 +1)$ and constants $f_{min},R>0$ and consider the Darcy flow model \eqref{Darcy} with some known, smooth source $g:\bar {\mathcal O}\to (0,\infty)$. Then there exist some constants $C>0$, $\gamma \in (0,2)$ and a sequence of statistical estimators $(\hat f_N:N\ge 1)$ such that $\hat f_N$ is the output of an algorithm with deterministic runtime $O(N^\gamma)$ and satisfies the statistical convergence rate
\[\mathbb E_{f_0}^N\big[\|\hat f_N-f_0\|_{L^2}^2 \big] \le C N^{-\frac{2\alpha-2}{2\alpha+2+d}},\]
for all $f_0\in H^\alpha(\mathcal O)$ with $\inf_{x\in\mathcal O} f_0(x)\ge f_{min}$ and $\|f_0\|_{H^\alpha}\le R$, and where $\mathbb E_{f_0}^N$ denotes the expectation under the law of $Z^{(N)}$.
\end{thm}

Theorem \ref{darcy-main} follows from Theorem \ref{thm-darcy-polytime} below, where an explicit construction of $\hat f_N$ is given. The estimator arises as a high-dimensional wavelet discretization of the plug-in M-estimator from \eqref{plugin-1} with carefully chosen (quadratic) Sobolev penalties. It will be shown later that a modification of the procedure can also achieve adaptive convergence rates, see Theorem \ref{thm-adaptive}. Compared to previous computational guarantees for high-dimensional MCMC in the Darcy model (cf.~Theorem 5.3.6 in \cite{nickl2023bayesian}), our guarantees here do not rely on any form of warm-start initialisation, and thus are `global'. Moreover, the run-time guarantees given here do \textit{not} rely on high-probability arguments to assert that the likelihood is `probably log-concave', such that we are able to significantly relax regularity and dimensionality assumptions under which polynomial-time statistically consistent estimators can be devised. Previous results required $\alpha\ge 22$ and $d=\{2,3\}$, while Theorem \ref{darcy-main} is valid as long as $\alpha >d/2+1$--see Section \ref{sec-discussion} for further discussion. The method is illustrated in Figure \ref{fig:plugin-intro}.

%The estimation procedure which achieves the bounds in Theorem \ref{darcy-main} is a novel class of estimators which we call \textit{plug-in M-estimators},

%The estimators arise as \emph{plug-in \(M\)-estimators}, obtained through a convex relaxation of the nonlinear least-squares criterion. They can be viewed as computationally tractable surrogates for the likelihood- and posterior-based procedures discussed below.

\subsection{Generalized stability estimates}\label{sec-generalized-stability-motivation}

The majority of existing proof strategies for statistical convergence rates in non-linear inverse problems \cite{NVW18,MNP21,nickl2023bayesian, vollmer2013posterior} works via showing a consistency result in `forward risk' $\|\mathcal G(\hat f_N)- \mathcal G(f_0)\|_{L^2}$ equivalent to the statistical Hellinger distance $h(f_1,f_2)$ \cite{GN19}, and then utilizing \textit{stability estimates} (for Darcy flow, see for instance \cite{R81,itokunisch,BCDPW17}) of the form
\begin{equation}\label{classical-stability}
    \|f-f_0\|_{L^2}\lesssim \tnorm{\mathcal G(f)-\mathcal G(f_0)}, 
\end{equation}
for a suitable norm $\tnorm{\cdot}$, to prove convergence of $\hat f_N$ towards $f_0$. However, a fundamental limitation for this proof strategy seems that the estimator for $\mathcal G(f_0)$ is required to take values within the nonlinear range $\{\mathcal G(f):f\in\mathcal F \}$ of $\mathcal G$, which may be difficult to obtain in polynomial-time. A key ingredient to proving the statistical convergence guarantees for generalized M-estimators are so-called \textit{generalized stability estimates} which we introduce---stability properties of the defining equation $L_fu=g$ rather than the operator $\mathcal G$, which can be formulated even for $f\notin \mathcal F$ and $u\notin \{\mathcal G(f):f\in\mathcal F \}$ lying, respectively, outside of the domain and range of $\mathcal G$. In a simplified form, these stability estimates entail that for any sufficiently smooth pair of functions $(f,u)$,
\[ \|f-f_0\|_{L^2}\lesssim  \tnorm{u-\mathcal G(f_0)}_1+  \tnorm{L_fu -g}_2, \]
for suitable norms $\tnorm{\cdot}_1,\tnorm{\cdot}_2$. It turns out that such estimates can be obtained both in the Darcy flow model (Proposition \ref{darcy-generalized-stability}) and the Schr\"odinger model (Proposition \ref{schroed-generalized-stability}) via relatively simple extensions of existing PDE proof techniques.  Note that the second term vanishes whenever the PDE constraint $L_fu=g$ is exactly fulfilled, thus recovering the classical form \eqref{classical-stability} of stability of $\mathcal G$.

%In the case of Darcy flow, the generalized stability estimates will be shown in  Proposition \ref{darcy-generalized-stability} that for any sufficiently smooth pair of functions $(f,u)$, it holds that 
%\[ \|f-f_0\|_{L^2}\lesssim  \|u-\mathcal G(f_0)\|_{H^2} + \| \nabla \cdot (f\nabla u) -g \|_{L^2}. \]
%Both for Darcy flow and the Schr\"odinger problem, these estimates can be achieved via simple extensions of PDE arguments already present in the literature. 

 %These generalized stability properties are not specific to the Darcy problem, but can also be shown in the Schr\"odinger equation (see Proposition \ref{schroed-generalized-stability}); and one can expect them to hold in a larger class of non-linear PDE inverse problems, which we however deferred to future work.

\begin{figure}
    \centering
    \includegraphics[width=1\linewidth]{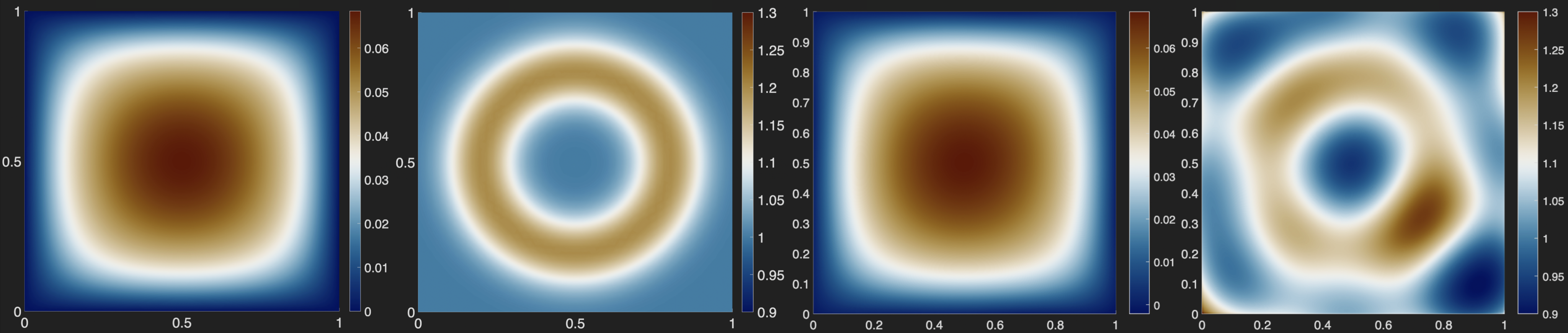}
    %\caption{Numerical illustration for the proposed plug-in estimation procedure in \textsc{MATLAB} based on $N=3000$ samples. Left two panels: Ground truth solution $\mathcal G(f_0)$ and diffusivity $f_0$. Right two panels: preliminary estimator $\hat u_N$ (here computed via GP regression) and plug-in estimator $\hat f_N$. Noise level was set to $0.0005$ to match the scale of the signal $\mathcal G(f_0)$. Runtimes are well below 1 minute.}

\caption{Numerical illustration of the proposed plug-in estimation procedure for the Darcy flow inverse problem. 
Left to right: ground truth solution $u_0=\mathcal G(f_0)$ and coefficient $f_0$, preliminary estimator $\hat u_N$ obtained by a smoothing regression step, and plug-in estimator $\hat f_N$, based on $N=3000$ noisy measurements.}
    
    \label{fig:plugin-intro}
\end{figure}

\subsection{Warm-start initialisation for MCMC}\label{sec-intro-warmstart}

Our theory has implications for the polynomial-time initialisation of Markov Chain Monte Carlo algorithms \cite{CRSW13} for sampling from the high-dimensional posterior distribution 
\[d\Pi(f|Z^{(N)}) \propto \exp(\ell_N(f))d\Pi(f)\]
in `Bayesian inversion' which has been widely studied in recent years \cite{S10,nickl2023bayesian}. In a line of work \cite{NW22, altmeyer2022polynomial, BN24} emanating from ideas developed in \cite{NW22}, a strategy for providing polynomial-time mixing guarantees for Langevin-type MCMC algorithms was given. In particular, it was proven that in models such as the Darcy flow and the Schr\"odinger model, the log-likelihood (and thus log-posterior) surface can be expected to be locally log-concave on some region $\mathcal I$ around the ground truth $f_0$ (see Figure \ref{posterior-landscape}). Thus, when the posterior distribution concentrates around $f_0$ at a sufficiently fast rate, most of its mass will be contained a sub-region $\mathcal C\subseteq I$---such conditions can, in many models, be verified via nonparametric posterior contraction techniques \cite{GV07,NW22}. In contrast, away from the ground truth $f_0$, log-concavity \textit{cannot} in general be expected to hold (Figure \ref{posterior-landscape}).

\begin{figure}[h]\label{posterior-landscape}
\centering
\begin{tikzpicture}[
  x=1cm,y=1cm,
  every node/.style={font=\footnotesize},
  line/.style={line width=1.1pt},
  dashedline/.style={line width=0.8pt, dash pattern=on 3pt off 2pt}
]

% --- pastel colors ---
\definecolor{pastelgreen}{RGB}{153,213,148} % contraction
\definecolor{pastelred}{RGB}{252,141,98}    % log-concave / warm start
\definecolor{pastelviolet}{RGB}{204,173,226}% potential multimodality

% --- geometry (x-coordinates shifted by +0.70) ---
\def\xshift{0.70}
\def\ybase{0.80}      % baseline
\def\yshade{1.60}     % TOP of shaded bands
\def\xa{5.20+\xshift}         % left  outer (warm-start) boundary (now 5.90)
\def\xb{5.90+\xshift}         % left  inner (contraction) boundary (now 6.60)
\def\xc{7.40+\xshift}         % right inner (contraction) boundary (now 8.10)
\def\xd{7.80+\xshift}         % right outer (warm-start) boundary (now 8.50)
\def\xstart{1.00+\xshift}     % Leftmost x-coordinate for shading/baseline (now 1.70)
\def\xend{13.70+\xshift}      % Rightmost x-coordinate for shading/baseline (now 14.40)

% ===== Shaded regions (no outlines) =====
\fill[fill=pastelgreen, opacity=0.55, draw=none] (\xb,\ybase) rectangle (\xc,\yshade); % R1 (C)
\fill[fill=pastelred,   opacity=0.40, draw=none] (\xa,\ybase) rectangle (\xb,\yshade); % R2 (I-L)
\fill[fill=pastelred,   opacity=0.40, draw=none] (\xc,\ybase) rectangle (\xd,\yshade); % R2 (I-R)
\fill[fill=pastelviolet,opacity=0.35, draw=none] (1.00,\ybase) rectangle (\xa,\yshade);% R3 (M-L) - start at 1.00 for visual space
\fill[fill=pastelviolet,opacity=0.35, draw=none] (\xd,\ybase) rectangle (13.70,\yshade);% R3 (M-R) - end at 13.70 for visual space

% baseline
\draw[line] (1,\ybase) -- (13.7,\ybase);

% === Likelihood curve: lifted without changing shape (x-coords shifted) ===
\begin{scope}[yshift=0.35cm]  % add a bit of vertical breathing room
  \draw[line]
    plot[smooth, tension=0.6] coordinates {
    (0.9+\xshift,3.4)  
      (1.6+\xshift,3.2)  
      (2.7+\xshift,3.3)
      (3.6+\xshift,3.2)
      (4.7+\xshift,2.72)  (5.4+\xshift,2.1)  (6.1+\xshift,1.57)
      (6.6+\xshift,1.4)  (7.0+\xshift,1.44) (7.4+\xshift,1.65)
      (8.0+\xshift,2.2)  (8.7+\xshift,2.6)  (9.4+\xshift,2.7)
      (10.2+\xshift,2.5) (11.0+\xshift,2.9) (11.9+\xshift,3.3) (12.5+\xshift,3.45)
    };
\end{scope}

% y-label
\node[anchor=east,fill=white,inner sep=1pt] at (13.4+\xshift,2.8) {$-\log d\Pi(f|Z^{(N)})$}; % Shifted label

% subtle vertical guides (tops nudged up to match curve lift, x-coords shifted)
\draw[dashedline] (\xa,\ybase) -- (\xa,2.65);
\draw[dashedline] (\xb,\ybase) -- (\xb,2.05);
\draw[dashedline] (\xc,\ybase) -- (\xc,2.02);
\draw[dashedline] (\xd,\ybase) -- (\xd,2.4);

% ===== Zone labels on bands (computed mid-height) =====
\pgfmathsetmacro{\ymid}{(\ybase+\yshade)/2}
% The coordinates here already use the shifted \xa, \xb, \xc, \xd
\node at ({(\xa+1.00)/2},\ymid) {$\mathcal M$}; % Need to explicitly shift start
\node at ({(\xa+\xb)/2},\ymid) {$\mathcal I$ };
\node at ({(\xb+\xc)/2},\ymid) {$\mathcal C$};
\node at ({(\xc+\xd)/2},\ymid) {$\mathcal I$};
\node at ({(\xd+13.70)/2},\ymid) {$\mathcal M$}; % Need to explicitly shift end

% ticks and labels on baseline
\pgfmathsetmacro{\fc}{(\xb+\xc)/2} % Center of C
\draw[line] (\fc,\ybase) -- (\fc,\ybase-0.1);
\node at (\fc,0.45) {$f_0$};

% bottom-right text
\node[anchor=east,fill=white,inner sep=1pt] at (13.8,.45)
  {$\mathcal{F}\simeq \mathbb{R}^D\ (D \text{ large})$};

% ===== Minimal legend (boxless, small) =====

% ===== Horizontal legend ABOVE the curve (shifted to center) =====
\begin{scope}[yshift=-.5cm] 
\node[
  anchor=south,
  font=\scriptsize,
  align=center,
  draw=black!30,        % subtle grey border
  fill=white,
  line width=0.3pt,
  rounded corners=1pt,
  inner sep=4pt,
  text depth=0pt,
  text height=1.5ex
] at (7.35,4.55) {      % Centered at 7.35 (midpoint of 1.00 to 13.70)
  \tikz{\fill[pastelgreen,opacity=0.7,draw=none]
      (0,0) rectangle (0.35,0.15);} $\mathcal C$: posterior contraction \quad
  \tikz{\fill[pastelred,opacity=0.7,draw=none]
      (0,0) rectangle (0.35,0.15);} $\mathcal I$: warm-start region (log-concave) \quad
  \tikz{\fill[pastelviolet,opacity=0.7,draw=none]
      (0,0) rectangle (0.35,0.15);} $\mathcal M$: potential multimodality
};
\end{scope}

% === Parabolic approximation (x-coords shifted) ===
% Center of parabola was 6.8, now 6.8 + 0.70 = 7.5
%\draw[gray!50, line width=0.8pt, dash pattern=on 4pt off 3pt, domain=4.0:11.03, samples=100] % Adjusted domain
%  plot (\x, {0.23*(\x-7.5)^2 + 1.7});

%\node[anchor=east,fill=white,inner sep=1pt,text=gray!90] at (10.95+\xshift,4.1) {$\tilde J(f)$}; % Shifted label
\end{tikzpicture}
\caption{Schematic representation of the log-posterior landscape. Often (e.g.~in the Darcy flow and Schr\"odinger models) the negative log-posterior is convex with high probability \cite{NW22, nickl2023bayesian} within some region $\mathcal I$ containing $f_0$, while it may be multimodal in $\mathcal M$. Cold-start MCMC algorithms may thus suffer from computational difficulties arising from multimodality or `free energy barriers' \cite{BMNW23}.} \end{figure}
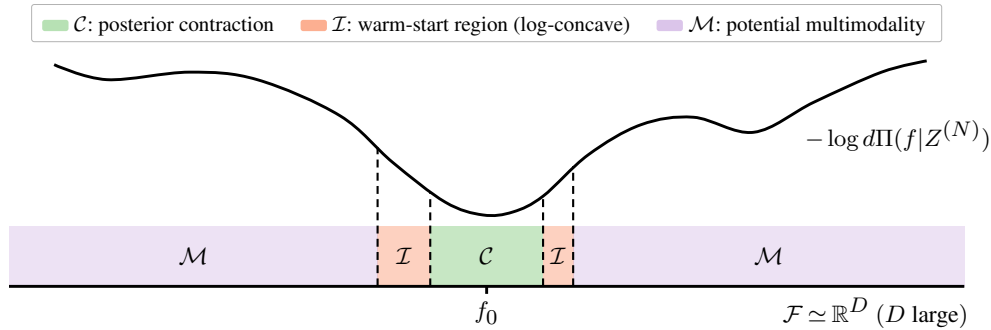

In such settings, one can decompose the posterior computation problem into a (i) \textit{global initialisation problem} consisting of the task of finding a starting point for the MCMC chain within $\mathcal I$ and (ii) the \textit{a local computation} problem after successful `warm-start' initialisation. In the Darcy problem, the initialisation problem (i) had been open so far, while the local computation problem (ii) has been addressed in \cite{nickl2023bayesian}---albeit under strong regularity assumptions, see Section \ref{sec-darcy-MCMC} for details. In fact, it is shown in \cite{BMNW23} that upon cold-start initialisation, high-dimensional non-log-concave sampling can be exponentially slow due to so-called `free energy barriers' in related regression models. For Darcy flow, when $f_0\in H^\alpha(\mathcal O)$ is sufficiently smooth and strictly positive, it was shown in \cite{nickl2023bayesian} (see Section 5.3 of that reference) that the posterior distribution is, with high probability, log-concave in a region of radius $D^{-8/d}$ (up to log-factors). Here, one employs a $D$-dimensional discretized representation $\R^D\ni \theta\mapsto f_\theta$ of the parameter $f\in\mathcal F$, with growing dimension $D\lesssim N^{\frac{d}{2\alpha+d}}$ as $N\to\infty$. Thus, roughly speaking, one can formulate the initialisation task for Darcy flow as finding a suffiently good estimator $\vartheta_{init}\in \R^D$ in polynomial-time such that
\begin{equation*}
    \begin{split}P_{f_0}^N(\vartheta_{init}\in\mathcal I)\xrightarrow{N\to\infty} 1,~~\text{with}~~
        %\|\vartheta_{init} - \theta_{0,D}\|_{\R^D}&\le (N^{\frac{8}{2\alpha+d}}\log N)^{-1},\\
        \mathcal I:=& \big\{\theta: \|\theta-\theta_{0,D}\|_{\R^D} \le N^{-\frac{8}{2\alpha+d}}/\log N \big\},
    \end{split}
\end{equation*}
where $\theta_{0,D}$ denotes the truncated series representation of $f_0$. We refer to Section \ref{sec-darcy-MCMC} for details. Note that the size of the region $\mathcal I$ is shrinking as $N,D$ tend to infinity.

\begin{thm}[Posterior computation for Darcy flow, informal]
Consider the Darcy problem for $d\le 3$, with $f_0 \ge f_{min}>0$ and $f_0\in H^\alpha(\mathcal O)$ for $\alpha$ sufficiently large. Then, a polynomial-time initialiser $\vartheta_{init}$ into the region $\mathcal I(N,D)$ exists. Moreover, for any polynomial precision level $\eps\ge N^{-P}$, with probability tending to $1$, the posterior mean $\bar \theta_N:=\int \theta d\Pi(\theta|Z^{(N)})$ arising from a Gaussian Matern-type series prior on $\theta \in\R^D$ (detailed in Section \ref{sec-darcy-MCMC}) can be computed up to precision $\eps$ using polynomially many $O(N^\gamma),~\gamma>0$, iterations of a high-dimensional MCMC algorithm initialized at $ \vartheta_{init}$.
\end{thm}

Note that the computation of the posterior mean $\bar \theta_N\in \R^D$ is a non-trivial task, since it amounts to computing a $D$-dimensional integral with respect to the non-log concave posterior distribution. In Section \ref{sec-darcy-MCMC}, we will explain in detail the MCMC algorithm which can be used for posterior computation, which is a modified version of the `Unadjusted Langevin Algorithm' (ULA), a construction which follows \cite{NW22, nickl2023bayesian}. The result is then proved by combining our theory with local computation results developed in \cite{NW22, nickl2023bayesian}. One can similarly also give polynomial-time guarantees for the computation of Maximum a Posteriori (MAP) estimators, but we refrain from doing so here.

\subsection{Discussion and comparison with related literature}\label{sec-discussion}
\subsubsection{Statistical convergence guarantees}

There have been several lines of work investigating the statistical convergence properties of regularization-based and Bayesian inference procedures in non-linear inverse problems. The papers \cite{NVW18} and \cite{MNP21} gave general recipes for deriving convergence rates for MAP estimators as well as posterior contraction rates with Gaussian priors, see also \cite{AW24, GN19, vollmer2013posterior} for related results. Another line of work has investigated the nonparametric \textit{Bernstein-von Mises} phenomenon, aiming to establish asymptotic normality of the posterior distribution; we refer to \cite{N20,MNP21-AOS} as well as \cite{nickl2023bayesian} for more references. While such results have been well-understood in the Schr\"odinger equation model, in the Darcy model fundamental difficulties arise relating to the range properties of the `information operator' $D\mathcal G_{f_0}^* D\mathcal G_{f_0}$ \cite{nickl2023bayesian}. The best known convergence rates derived in the above-mentioned works, for the Darcy model, are given by
\begin{equation*}
    \begin{split}
        \E [\|\hat f_N - f_0\|_{L^2}^2]  &\lesssim N^{-\frac{2(\alpha-1)}{2(\alpha +1) +d}}~~~~~~~\text{(inverse problem)},\\
        \E [\|\mathcal G(\hat f_N)- \mathcal G(f_0)\|_{L^2}^2] &\lesssim N^{-\frac{2 (\alpha + 1 )}{2(\alpha + 1) +d}}~~~~~~~\text{(forward problem)},
    \end{split}
\end{equation*}
where $\alpha$ represents the smoothness of $f_0$. For the forward problem, the rate can be shown to be minimax-optimal \cite{NVW18}. These rates are matched by all the methods proposed in the current paper---see the Theorems \ref{darcy-main} and \ref{Darcy-rate}, \ref{thm-darcy-polytime} below. For the Schr\"odinger model, our methods achieve minimax-optimal convergence rates both for the forward and the inverse problem, see Theorem \ref{Schroed-rate} below.

The paper \cite{KSV24} pursues a related plug-in principle in a Bayesian framework: it first derives posterior contraction for forward quantities (in particular $u$ and its derivatives) and then recovers $f$ via an explicit inversion step based on classical numerical ideas \cite{R81,richter1981numerical}, which in practice amounts to solving a transport PDE. Their development largely focuses on $d=2$ and hinges on structural assumptions such as $\mathcal G(f_0)$ lying in an iterated eigenspace of the prior covariance operator. In contrast, our approach avoids explicit inversion via transport/characteristics: we formulate the recovery of $f$ through a PDE-penalty and leverage generalized stability, which yields nested quadratic optimization problems. This yields explicit, non-asymptotic $o(N^2)$ runtime guarantees for the discretized estimators (cf.~Theorem~\ref{thm-darcy-polytime}), whereas obtaining comparable end-to-end complexity guarantees for transport-based inversion would require a separate analysis of the full numerical pipeline with hyperbolic PDEs.

\subsubsection{Existing computational guarantees and MCMC mixing times}

There has been an extensive literature developing on high-dimensional Bayesian computation in recent years. In particular, there has been a surge of interest in high-dimensional log-concave sampling \cite{D17,DM18} and infinite-dimensional mixing of MCMC algorithms \cite{HSV14}. However, in the present context non-linear inverse problems, the `Gibbs measure' $\Pi(\cdot|Z^{(N)})$ is non-convex and computational guarantees have remained a core challenge. As discussed in Section \ref{sec-intro-warmstart}, a line of work starting from the paper \cite{NW22} (see \cite{altmeyer2022polynomial,BN24, nickl2023bayesian}) has derived several theoretical computational guarantees for Langevin-type Markov Chains upon warm-start. In the Schr\"odinger model, even a global polynomial-time computation guarantee was shown in \cite{NW22}; albeit using completely different proof techniques from ours. The techniques there rely on so-called `gradient stability' arguments for the linerization $D\mathcal G$ of the forward operator to assert local log-concavity of the expected likelihood function, and then utilize concentration-of-measure arguments from empirical process theory \cite{talagrand2005generic} to assert the local log-concavity of $\ell_N$ with high probability. While this approach achieved the first \textit{qualitative} polynomial-time results, it seemed rather difficult to make the exponents in the runtime guarantees explicit. Moreover, the computational feasibility seems to be reliant on the occurrence of a high-probability event, which may only have significant probability for large $N$. Thirdly, due to the aforementioned concentration-of-measure arguments, extremely strong regularity assumptions (e.g., $d\le 3$ and $\alpha\ge 22$) are required---this can be seen in our polynomial-time MCMC Theorem \ref{thm-darcy-MCMC}, which builds upon those references. One of the main goals of this work was to relax those stringent assumptions. In line with this, the computational runtime of our statistical estimators have explicit (sub-quadratic) exponents and is \textit{deterministic}, taking a step towards more practical algorithmic guarantees. However, it is not yet clear whether our estimators can achieve (semi-parametrically) efficient uncertainty quantification \cite{V98,nickl2023bayesian}, see Section \ref{sec-further-dis} for more comments.

We also mention here the classical regularization-based literature which has considered (penalized) least squares schemes---typically in a context of deterministic noise, see, e.g.~\cite{hofmann2007convergence, neubauer1989tikhonov,HNS95,H17,KNS08}. Similarly to the statistical setting here, the algorithmic guarantees developed in this literature are typically limited by the non-convexity of the involved least-squares functionals, and are thus limited to local `warm-start' guarantees; we refer to \cite{KNS08} for an overview.

\subsubsection{PINN, PDE-regularized methods, `all-at-once' formulations}

The idea of penalizing the deviation from the `physics' prescribed by the defining equation $L_f=g$ is, by now, common in many key methodologies in applied mathematics and scientific computing. Related to our Method I \eqref{relax-step2} are the `all-at-once' formulations of inverse problems \cite{kaltenbacher2016regularization} which also utilize the idea of augmenting the optimization space. The most notable recent methodology using this are physics-informed neural networks (PINN) \cite{RPK19, doumeche2025convergence}. While these methodologies were initially aimed at solving (potentially high-dimensional) PDEs, they have recently also been extended to the solution of inverse problems, we refer the reader to \cite{PINN-review22} for an overview. Previous to PINNs, similar ideas have appeared e.g.~in optimal control of PDEs \cite{troltzsch2010optimal} and in PDE-regularized regression \cite{sangalli2021spatial}, albeit usually not targeting inverse problems. For the above methodologies, some statistical consistency theory has gradually developed over recent years---see, e.g., \cite{PDEregularizedConsistency}, \cite{doumeche2025convergence}, \cite{lu2021PINN}. However, at least for PINN-type methodologies, the question of rigorous statistical-algorithmic guarantees yet seems out of reach---due to the highly non-linear neural network parameterization of functional parameters there---and is a highly ambitious research future direction.

\subsubsection{Non-linear PDEs}
We conclude the introduction with a brief discussion of the applicability of our theory to even more complicated non-linear inverse problems. Methodologically, our methods can be applied to any inverse problem given by $\mathcal L_f[u] = g$, including the case when the defining PDE is \textit{non-linear}, in the sense that $L_f$ is a non-linear differential operator. In cases where the mapping $(f,u)\mapsto L_fu$ is still linear in the $f$-variable, both steps of the plug-in estimator constitute quadratic optimisation problems, and thus are computationally tractable. This includes the case of inferring the interaction potential in McKean-Vlasov equations \cite{McKean-Vlasov-RayPavliotis}, and inferring the non-linear `reaction function' in reaction-diffusion type equations arising in physics, chemistry and biology. In these models, traditional likelihood-based and Bayes methods are even more expensive that in the elliptic PDE examples discussed here, and it would be of significant interest to derive statistical guarantees via generalized stability estimates for plug-in methods in these non-linear models. In contrast, there is a class of non-linear inverse problems where the parameter $f$ to be inferred is not defined on the same spatial (or spatio-temporal) domain as the regression measurements, whence it seems challenging to directly apply our ideas. Notably, this includes inference problems on the initial condition in data assimilation of  (deterministic) dynamical systems \cite{NT24,nickl2024bernstein} and non-linear X-ray transform problems \cite{paternain2023geometric} with boundary measurements.

\subsection*{Some notation}
We fix some notation used throughout. For an open set $\mathcal O\subseteq \R^d$, we denote by $H^\alpha(\mathcal O)$
the standard $L^2(\mathcal O)$–Sobolev spaces of functions with square-integrable
weak derivatives up to order $\alpha$. Moreover, $C(\mathcal O)$ and
$C(\overline{\mathcal O})$ denote the spaces of continuous functions on
$\mathcal O$ and $\overline{\mathcal O}$, respectively, equipped with the supremum norm.
For integers $k\ge 1$, $C^k(\mathcal O)$ denotes the space of $k$-times continuously
differentiable functions for which the supremum norm of all partial derivatives up to order $k$ is finite,
\[
\|u\|_{C^k(\mathcal O)} := \max_{|\beta|\le k}\sup_{x\in\mathcal O}|\partial^\beta u(x)|.
\]
For non-integer $\alpha>0$, $C^\alpha(\mathcal O)$
denotes the usual Hölder spaces. Finally, we write $\nabla$, $\nabla\cdot$, and
$\Delta$ for the gradient, divergence, and Laplacian operators.

\section{Main results for PDE models}\label{sec:pde-results}

This section contains our main findings for two prototypical PDE model examples, the \textit{Darcy problem} and the \textit{steady-state Schr\"odinger equation}. In both cases, the forward map $\mathcal G$ arises as a parameter-to-solution operator of an elliptic PDE where the infinite-dimensional statistical parameter space $\mathcal F$ is assumed to consist of Sobolev-regular functions. Throughout, let $\mathcal O\subseteq \R^d$ denote a bounded, smooth (open) domain with smooth boundary $\partial \mathcal O$, assumed without loss of generality to have unit Lebesgue measure, $\text{vol}(\mathcal O)=1$. For integer $\alpha>0$, $f_{min}>0$ and $R>0$, we denote
\begin{equation}\label{F-def}
    \mathcal F= \mathcal F(\alpha,f_{min}, R) := \big\{ f\in H^\alpha (\mathcal O) : f\ge f_{min},~ \|f\|_{H^\alpha(\mathcal O)}\le R \big\}.
\end{equation}
In the Darcy model, the lower bound $f_{min}>0$ ensures uniform ellipticity, while the smoothness level $\alpha$ will determine the convergence rates achieved. The ensuing theory requires $\alpha>d/2$ in order to guarantee sufficient regularity for the PDE analysis---in this case, the Sobolev embedding
$H^\alpha(\mathcal O)\hookrightarrow C^\eta(\mathcal O)$ for any $\eta\in(0,\alpha-d/2)$
ensures that functions in $\mathcal F$ admit pointwise representatives. Throughout this section, we consider the random design regression model
\begin{equation}\label{PDE-data}
    Y_i = \mathcal G(f)(X_i) + \varepsilon_i, \qquad
    X_i \sim^{i.i.d.} \mathrm{Unif}(\mathcal O), \quad
    \varepsilon_i \sim^{i.i.d.} N(0,1),\quad 1\le i\le N,
\end{equation}
where $(X_i)_{i=1}^N$ and $(\varepsilon_i)_{i=1}^N$ are independent. While our techniques extend to more general sampling distributions $X_i\sim P_X$
and to broader noise classes (e.g.\ sub-Gaussian $\varepsilon_i\in\R$),
we restrict attention to the present setting for clarity.

Guided by the intuition developed in Section~\ref{sec-motivation},
we establish theoretical guarantees for both the PDE-penalized M-estimators defined in \eqref{relax-step2} and the corresponding plug-in estimators from \eqref{plugin-1}; see Theorem~\ref{Darcy-rate}.
A central result is Theorem~\ref{thm-darcy-polytime}, which shows that a high-dimensional wavelet discretization of the plug-in estimator attains the best currently known statistical convergence rates while remaining computable in sub-quadratic time. Further computational implications of these results, including polynomial-time
sampling guarantees, are discussed in Section~\ref{sec-darcy-MCMC}, while adaptive rates are derived in Section~\ref{sec-adaptive}. These results focus on the Darcy model. Analogous statements for the steady-state Schr\"odinger equation are discussed more briefly, due to space constraints, in Section~\ref{sec-schroe}.

\subsection{Darcy flow: infinite-dimensional estimators}\label{sec-darcy}

Let $g\in C^\infty(\mathcal O)$ be fixed. For any conductivity $f:\mathcal O\to(0,\infty)$
satisfying $\inf_{x\in\mathcal O} f(x)>0$ and $f\in C^{1+\eta}(\bar{\mathcal O})$
for some $\eta>0$, the Dirichlet problem
\begin{equation}\label{darcy}
  \begin{cases}
\nabla\cdot(f\nabla u)=g & \text{in }\mathcal O,\\
u=0 & \text{on }\partial\mathcal O,
\end{cases}  
\end{equation}
admits a unique classical solution $u_f\in C^2(\mathcal O)\cap C(\bar{\mathcal O})$
by standard elliptic theory \cite[Theorem~6.14]{GT98}.
We denote by $\mathcal G(f)=u_f$ the associated parameter-to-solution map. Whenever $\alpha>d/2+1$ and $f\in\mathcal F(\alpha,f_{\min},R)$,
the Sobolev embedding
$H^\alpha(\mathcal O)\hookrightarrow C^{1+\eta}(\mathcal O)$ for
$\eta\in(0,\alpha-d/2-1)$
ensures that $\mathcal G(f)$ is well-defined for all $f\in\mathcal F$.

The following theorem considers two estimation procedures for the Darcy model.
The first is the PDE-penalized M-estimator introduced in \eqref{relax-step2}, with suitably chosen Tikhonov penalties. Writing $L_f u=\nabla\cdot(f\nabla u)$, we define the PDE-penalized M-estimator as any measurable minimizer of the functional
\begin{equation}\label{darcy-PDE-penalized}
    \begin{split}
J(u,f)
=\frac{1}{N}\sum_{i=1}^N\big[Y_i-u(X_i)\big]^2
+ \lambda_N^2\|L_f u-g\|_{L^2(\mathcal O)}^2 + \mu_N^2\Big(\|u\|_{H^{\alpha+1}(\mathcal O)}^2
+ \|f\|_{H^\alpha(\mathcal O)}^2\Big),
\end{split}
\end{equation}
over $(u,f)\in H^{\alpha+1}(\mathcal O)\times H^\alpha(\mathcal O)$,
with regularization parameters
\begin{equation}\label{Darcy-hyperparams}
\lambda_N = N^{-\frac{2}{2(\alpha+1)+d}},
\qquad
\mu_N = N^{-\frac{\alpha+1}{2(\alpha+1)+d}}.
\end{equation}
The second estimator is the two-stage plug-in M-estimator from \eqref{plugin-1}. Specifically, it is given by
\begin{equation}\label{darcy-plugin}
\begin{split}
\hat u_N
&\in \argmin_{u\in H^{\alpha+1}(\mathcal O)}
\Big\{\frac{1}{N}\sum_{i=1}^N\big[Y_i-u(X_i)\big]^2
+ \mu_N^2\|u\|_{H^{\alpha+1}(\mathcal O)}^2\Big\},\\
\hat f_N
&\in \argmin_{f\in H^\alpha(\mathcal O)}
\Big\{\|\nabla\cdot(f\nabla \hat u_N)-g\|_{L^2(\mathcal O)}^2
+ \nu_N^2\|f\|_{H^\alpha(\mathcal O)}^2\Big\},
\end{split}
\end{equation}
where $\mu_N$ is as in \eqref{Darcy-hyperparams} and $\nu_N = N^{-\frac{\alpha-1}{2(\alpha+1)+d}}.$ Note that in \eqref{darcy-plugin}, both optimization problems are strictly convex (in fact, quadratic) and therefore admit unique
solutions.

Importantly, both formulations allow the estimator of $u_0=\mathcal G(f_0)$
to lie outside the range of the forward operator $\mathcal G$, and they are formulated as unconstrained problems over the \textit{linear spaces} $H^\alpha(\mathcal O)$ and $H^{\alpha+1}(\mathcal O)$, respectively, in particular neither enforcing the positivity constraint of $f>0$ nor the range constraint $u\in \{\mathcal G(f):f\in\mathcal F\}$. A priori, it is not evident that such a procedure should achieve the same statistical guarantees as likelihood-based approaches
\cite{NVW18,nickl2023bayesian, GN19}.
As explained in Section \ref{sec-generalized-stability-motivation}, this is owing to the generalized
stability properties of the Darcy problem which we will state below in Lemma \ref{darcy-generalized-stability}. We now state the main convergence result for both estimators, which will be proven in Section \ref{sec-darcy-infdim-proof}.

\begin{thm}\label{Darcy-rate}
    Let $d\ge 1$ and suppose $\mathcal O\subseteq \R^d$ is a smooth, bounded domain. Let $\alpha> d/2+1$ and $\alpha \ge 3$. Fix constants $f_{min}>0,~R>0$. Suppose that the data (\ref{PDE-data}) arise from some $f_0\in \mathcal F(\alpha,f_{min},R)$ and let $(\hat u_N,\hat f_N)$ denote either the generalized M-estimator from \eqref{darcy-PDE-penalized} or the plug-in M-estimator from \eqref{darcy-plugin}. There exists $C>0$ only depending on $\mathcal O,d, \alpha,f_{min},R$ such that for all $N\ge 1$,
    \begin{align}
            \E_{f_0}^N\big[  \|\hat u_N - \mathcal G(f_0)\|_{L^2(\mathcal O)}^2\big] &\le C N^{-\frac{2(\alpha+1)}{2(\alpha+1)+d}}, ~~\text{and}\label{Darcy-forward}\\
            \E_{f_0}^N \big[ \|\hat f_N - f_0\|_{L^2(\mathcal O)}^2 \big]&\le C N^{-\frac{2(\alpha-1)}{2(\alpha+1)+d}}.\label{Darcy-inverse}
    \end{align}
\end{thm}

Two convergence rates are established in Theorem~\ref{Darcy-rate}.
The first concerns the forward problem of estimating $\mathcal G(f_0)$
and coincides with the minimax rate for estimating $(\alpha+1)$-smooth regression
functions. It is shown in \cite{NVW18} that this rate is optimal, in a minimax sense,
over the model class $\{u_f : f\in\mathcal F(\alpha,f_{\min},R)\}$. The second rate addresses the inverse problem of recovering the conductivity $f_0$.
At a heuristic level, this rate corresponds to the minimax estimation rate in a $2$-ill-posed inverse problem where the observed signal is $\alpha+1$-smooth. Indeed, the inverse map $\mathcal G^{-1}:u\mapsto f$ is known to
satisfy a $2$-ill-posed stability estimate: for sufficiently smooth and uniformly
elliptic coefficients $f_1,f_2$, one has
\[
\|f_1-f_2\|_{L^2(\mathcal O)}
\;\lesssim\;
\|u_{f_1}-u_{f_2}\|_{H^2(\mathcal O)},
\]
see, for instance, Lemma~24 in \cite{NVW18}. Similar stability estimates appear in
earlier works, such as \cite{R81} for $L^\infty$--$C^2$ bounds and
\cite{itokunisch} for one-dimensional $L^2$ estimates. The central analytical insight underlying
Theorem~\ref{Darcy-rate} is that suitable stability estimates continue to hold
even when the PDE constraint is satisfied only approximately. The following lemma makes this \emph{generalized stability} phenomenon precise.

\begin{restatable}{lem}{darcygeneralizedstability}
\label{darcy-generalized-stability} Let $\mathcal O\subseteq \R^d$ be a smooth, bounded domain and let $f_{min}, g_{min}>0$. Suppose that for $i\in \{1,2\}$, we have functions 
\[ u_i \in H^2(\mathcal O),~~~f_i\in C^1(\mathcal O),~~~g_i\in L^2(\mathcal O)~~~\text{with}~~~\nabla\cdot(f_i\nabla u_i) = g_i~~~\text{on}~\mathcal O.\]
Assume that the first triplet $(u_1,f_1,g_1)$ additionally satisfies $u_1\in C^2(\mathcal O)\cap C(\bar {\mathcal O})$, $g_1\in C(\mathcal O)$ and 
\[\inf_{x\in \mathcal O}f_1(x) \ge f_{min}~~~~ \text{and}~~~~\inf_{x\in \mathcal O}g_1(x)\ge g_{min},~~~~\text{and}~~~~u_1|_{\partial \mathcal O}=0.\]
\noindent Then, for every $B>0$ there exists a constant $C>0$ solely depending on $f_{min},g_{min},\mathcal O$ and $B$ such that whenever $\|f_1\|_{C^1}\vee \|u_{1}\|_{C^2}\le B$, it holds that
\[ \|f_1-f_2\|_{L^2}\le C\big(  \|f_2\|_{C^1} \|u_{1}-u_{2}\|_{H^2} + \|g_1-g_2\|_{L^2}\big).  \]
\end{restatable}

A key feature of Lemma~\ref{darcy-generalized-stability} is the \emph{asymmetry} of its hypotheses. The main structural assumptions (uniform ellipticity, Dirichlet condition, and lower bound on $g_1$) are imposed only on the reference triplet $(u_1,f_1,g_1)$, while $(u_2,f_2,g_2)$ is unrestricted apart from basic regularity. Due to boundedness properties of the forward map $\mathcal G$, it turns out that the constant $B\ge \|f_1\|_{C^1}\vee \|u_{f_1}\|_{C^2}$ can be chosen uniformly over triplets $g_1=g$, $u_1=\G(f_0)$ and $f_1=f_0\in \mathcal F$ with $\alpha>d/2+1$ representing the data-generating ground truth. Moreover, we note that the proof of Theorem \ref{darcy-main} in fact yields stronger (sub-Gaussian) concentration inequalities for suitable
penalized empirical risk functionals---see the proof for details.

\subsection{Darcy flow: high-dimensional estimator}\label{sec-high-dim}

In this section, we present the high-dimensional estimator which underlies Theorem \ref{darcy-main}. The estimator arises from a high-dimensional wavelet frame discretization of the infinite-dimensional plug-in scheme \eqref{darcy-plugin}. Wavelets are chosen over alternative representation systems, such as Dirichlet Laplacian eigenfunctions, in order to accurately approximate functions in $H^{\alpha+1}(\mathcal O)$ up to the boundary $\partial\mathcal O$. The main result of this section is Theorem \ref{thm-darcy-polytime}, asserting that the estimator is computable in sub-quadratic runtime, while possessing the same statistical convergence properties as its infinite-dimensional counterpart from Theorem \ref{Darcy-rate}.

Before stating our main results, we clarify the following about the computation model used in this section.

\begin{rem}[Computation model]\normalfont
    The explicit complexity bounds in Theorem \ref{thm-darcy-polytime} count floating-point operations associated with matrix–vector products and linear system solves. We assume oracle access to point evaluations and $L^2(\mathcal O)$ inner products of wavelet basis functions and their derivatives, which may be precomputed to high precision. Moreover, we upper bound computational costs by that of dense linear algebra; exploiting additional sparse structure of $\Phi$ or $\Psi$ could further reduce runtime, but this is not pursued here. This computation model refines previously used notions of computation time, such as MCMC iteration complexity used in \cite{NW22} as well as in Section \ref{sec-darcy-MCMC} below, which presume precise oracle access to evaluations $\mathcal G(f)$ of the forward map.
\end{rem}

\paragraph*{Wavelet frames and preliminaries}

We introduce some notation and basic facts about the wavelet frames which will be used throughout. For a comprehensive treatment of this theory, we refer the reader to \cite{T08}. As before, let $\mathcal O\subseteq \R^d$ denote a smooth, bounded domain. For some $u>\alpha+1$, we denote by
\[
\big\{ \phi_{lk} : l\ge 0, 1 \leq k \leq N_l \big\},
\]  
a wavelet frame of $L^2(\mathcal O)$ consisting of wavelets with smoothness at least $C^u(\mathcal O)$ in the sense of Definition 5.25 in \cite{T08}; the existence of such a wavelet system is ensured by Theorem 5.51 of \cite{T08}. At `resolution level' $l \ge 0$, there are at most $N_l \lesssim 2^{l d}$ elements. Given any multi-scale vector
\[v= (v_{lk}: l\ge 0,~1\le k\le N_l),\]
we can then define the following sequence norms
\[ \|v\|_{h^\alpha}^2 := \sum_{l\ge 0}\sum_{k\le N_l} 2^{2l\alpha}v_{lk}^2, ~~\alpha\ge 0, \]
with $h^\alpha$ denoting the space of all sequences where the norm is finite. A key property  of the above wavelet frame is that $\|\cdot\|_{h^\alpha}$ characterizes the classical Sobolev norms. Indeed, for any $f\in H^\alpha(\mathcal O)$, it holds that 
\begin{equation}\label{frame-eq}
    \inf_v \|v\|_{h^\alpha} \simeq \|f\|_{H^\alpha(\mathcal O)}, 
\end{equation}
where the infimum runs over all admissible representations $f=\sum_{l,k} v_{lk} \phi_{lk}$, see Theorem 5.51 of \cite{T08} (with $p=q=2$). The infimum on the left hand side is necessary since the $(\phi_{lk})$ may not form a basis, whence the representation of $f$ may not be unique. However, there exists a bounded and linear operator $S: H^\alpha (\mathcal O) \to h^\alpha$ which selects a representation such that 
\[  \|S(f)\|_{h^\alpha} \simeq \|f\|_{H^\alpha(\mathcal O)},~~~~~ f= \sum_{l,k} S(f)_{lk} \phi_{lk}, \]
see Theorem 5.51 in \cite{T08}. We denote the selection operator for the $(l,k)$-th coefficient by 
\[S_{lk}: H^\alpha(\mathcal O) \to \R, ~~ S_{lk}(f) = S(f)_{lk}.\]
Throughout, we fix one such selection operator. We also need finite-dimensional truncations of the above wavelet spaces. For resolution level $J\ge 1$, we denote the subspace with frequencies up to level $J$ by
\[ L^2_J(\mathcal O) = \text{span}( \phi_{lk},~l\le J,~k\le N_l)\subseteq L^2(\mathcal O), \]
which has dimension 
\[ \text{dim}(L^2_J(\mathcal O)) \le p_J:= \sum_{l\le J}N_l \lesssim 2^{Jd},~~~ J\ge 1. \]
In slight abuse of notation, we will write $v\in \R^{p_J}$ for the Euclidean space of finite-resolution multi-scale vectors, and  
\[ \|v\|_{h^\alpha_J}^2:= \sum_{l\le J, k\le N_l} 2^{2l\alpha} v_{lk}^2,~~~~ \alpha\ge 0,~J\ge 1,~v\in\R^{p_J}. \]

\subsubsection{Construction of $\hat u_N$} Let $\mu_N = N^{-\frac{\alpha+1}{2(\alpha + 1)+d}}$ denote the minimax rate of estimation for $(\alpha+1)$-smooth functions. At sample size $N\ge 1$, we choose the resolution level $J=J_N$ such that
\[ 2^{Jd} \simeq N \mu_N^2 = N^{\frac{d}{2(\alpha +1) + d}}.\]
This choice balances the approximation error of the wavelet projection with the stochastic error of estimating $p_J\asymp 2^{Jd}$ coefficients for an $\alpha+1$-regular function. The estimator $\hat u_N$ will be given as 
\begin{equation}\label{u-hat-darcy-high-dim}
    \hat u_N := \sum_{l \leq J, k\le N_l} \hat \eta_{lk}\phi_{lk} \in H^{\alpha+1}(\mathcal O), 
\end{equation}
where $\hat \eta_N$ is the (unique) minimizer of the quadratic functional
\[
\frac{1}{N} \sum_{i=1}^N \Big( Y_i - \sum_{l \leq J, k\le N_l} \eta_{lk} \, \phi_{lk}(X_i) \Big)^2
+ \mu_N^2 \| \eta \|_{h^{\alpha+1}_J}^2~~~\text{over}~~~\eta\in \R^{p_J}.\]
In matrix form, we can write this as a high-dimensional penalized least squares problem
\begin{equation}\label{u-finite-dim}
    \hat \eta_N =\argmin_{\eta \in \R^{p_J}} \Big( \frac 1N \big\| Y- \Phi \eta \big\|_{\R^N}^2 +  \mu_N^2\eta^T \Lambda^{\alpha+1} \eta \Big),
\end{equation}
where $\|\cdot\|_{\R^N}$ denotes the Euclidean norm, $\Phi \in \R^{N\times p_J}$ is the `design matrix' with entries $\Phi_{i,(l,k)}=\phi_{lk}(X_i)$ 
and $\Lambda$ is the $p_J\times p_J$ diagonal matrix representing the $h^{\alpha+1}$-norms, 
\begin{equation}\label{lambda-matrix}
     \Lambda := \text{diag}\big(1\dots \underbrace{2^{2l}\dots 2^{2l}}_{N_l~\text{times}}\dots 2^{2J} \dots 2^{2J} \big).
\end{equation}
From the formulation, one sees immediately that \eqref{u-finite-dim} indeed has a unique minimizer $\hat \eta_N$---since $\Lambda$ is strictly positive definite---given by
\begin{equation}\label{hat-eta-darcy}
    \hat\eta_N = \big(\Phi^T\Phi +N\mu_N^2 \Lambda^{\alpha+1} \big)^{-1} \Phi^TY.
\end{equation}

Since $p_J\asymp 2^{Jd}\asymp N\mu_N^2 = N^{\frac{d}{2(\alpha+1)+d}}$, 
the cost of forming $\Phi^\top\Phi$ and solving the resulting linear system involving the inversion of the inversion of $(\Phi^T\Phi +\mu_N^2 \Lambda^{\alpha+1})$ (assuming dense linear algebra) is bounded by 
\[
O(Np_J^2+p_J^3)=O\big(N^{1+\frac{2d}{2(\alpha+1)+d}}+N^{\frac{3d}{2(\alpha+1)+d}}\big) =O\big(N^{1+\frac{2d}{2(\alpha+1)+d}}\big),
\]
the computational cost of the remaining matrix-vector operations are of smaller order.

\subsubsection{Construction of $\hat f_N$} Conditional on $\hat u_N$, the plug-in estimation step for calculating $\hat f_N$ is again given as a basis expansion
\begin{equation}\label{f-hat-darcy}
    \hat f_N := \sum_{l \leq J, k\le N_l} \hat \theta_{lk}\phi_{lk},
\end{equation}
where $\hat\theta_N$ is the solution to the following high-dimensional penalized least squares problem. With $J=J_N$ chosen as before, we write $\gamma \in \R^{p_J}$ for the frame coefficients of the source $g$,
\[ \gamma_{lk} = S_{lk}(g), ~~~\forall ~l\le J,~k\le N_l. \]
The estimator $\hat \theta_N$ is then given as 
\[ \hat \theta_N = \argmin_{\theta\in \R^{p_J}}  \| \gamma - \Psi \theta \|_{\R^{p_J}}^2 +  \nu_N^2 \theta^T \Lambda^\alpha \theta, \]
where $\nu_N= N^{-\frac{(\alpha-1)}{2(\alpha+1)+d}}$ is chosen as in the infinite-dimensional case, $\Lambda$ is given by \eqref{lambda-matrix}, and the matrix $\Psi \in \R^{p_J\times p_J}$ is given by the coefficients
\begin{equation*}
    \begin{split}
    \Psi_{(lk), (l'k')}&= S_{l'k'} \big(\nabla \cdot ( \phi_{lk} \nabla \hat u_N)\big) \\
    &= S_{l'k'}(\nabla \phi_{lk}\cdot \nabla \hat u_N) +S_{l'k'}(\phi_{lk}\Delta \hat u_N)\\
    &= \sum_{m\le J, j\le N_m} (\hat \eta_N)_{mj}  \big[  S_{l'k'} (\nabla \phi_{lk}\cdot \nabla \phi_{mj} ) + S_{l'k'} (\phi_{lk}\Delta \phi_{mj}) \big].
    \end{split}
\end{equation*}
Here, $\hat \eta_N$ is the estimator given by \eqref{hat-eta-darcy}. Note that constructing $\Psi$ requires access to wavelet-frame coefficients of products of basis functions and their derivatives, such as 
$S_{l'k'}(\nabla \phi_{lk}\cdot\nabla \phi_{mj})$ and $S_{l'k'}(\phi_{lk}\Delta \phi_{mj})$; 
we treat these as available oracle quantities (or as precomputable offline to high precision) within our computation model. Given $\Psi$, $\hat \theta_N$ possesses a closed-form expression
\begin{equation}\label{theta-hat-darcy}
    \hat\theta_N = \big(\Psi^T\Psi +\nu_N^2 \Lambda^{\alpha} \big)^{-1} \Psi^T\gamma, ~~~~~\hat\theta_N \in \R^{p_J}.
\end{equation}

\begin{thm}\label{thm-darcy-polytime}
    Consider the Darcy flow problem on a bounded, smooth domain $\mathcal O\subseteq \R^d$ with smooth and positive $g\in C^\infty (\mathcal O)$. Suppose that $\alpha>d/2 +1$ and $f_0\in\mathcal F(\alpha, f_{min},R)$ for some $f_{min},R>0$. There exist constants $C_1,C_2>0$ depending at most on $\alpha,d,\mathcal O,f_{min},R$ and an exponent $\kappa = \kappa (d,\alpha)<2$ such that for all $N\ge 1$ the estimators $\hat\eta_N, \hat u_N, \hat \theta_N, \hat f_N$ defined by \eqref{u-hat-darcy-high-dim}-\eqref{theta-hat-darcy} satisfy the convergence rates 
    \begin{equation}
        \begin{split}
        \E_{f_0}^N[ \| \hat u_N-u_0\|_{L^2}^2] \le C_1N^{-\frac{2(\alpha+1)}{2(\alpha+1)+d}},~~~~\E_{f_0}^N[ \| \hat f_N-f_0\|_{L^2}^2] \le C_1 N^{-\frac{2(\alpha-1)}{2(\alpha+1)+d}},
        \end{split}
    \end{equation}
    and the high-dimensional vectors $\hat \eta_N,\hat \theta_N\in \R^{p_J}$ are computable within $C_2 N^{\kappa}$ floating-point operations.
\end{thm}

Theorem \ref{thm-darcy-polytime} in particular implies Theorem \ref{darcy-main}. In fact, the explicit runtime exponent is given by 
\[ \kappa=\max \Big\{1+\frac{2d}{2(\alpha+1)+d}, \frac{3d}{2(\alpha+1)+d}\Big\} = 1+\frac{2d}{2(\alpha+1)+d}\]
Since $\alpha>d/2+1$, we have $\kappa<2$, yielding a uniformly sub-quadratic computational complexity. In particular, while the statistical convergence rates deteriorate with increasing physical dimension $d$, achieving the (in the forward problem, minimax-optimal) rates does not incur a super-quadratic computational cost.

\subsection{Polynomial-time Bayesian computation}\label{sec-darcy-MCMC}

We now state two results about polynomial-time computational guarantees for high-dimensional posterior sampling. The notion of polynomial-time computation used below differs from that
employed in Section~\ref{sec-high-dim}. The Bayesian literature typically measures computational complexity
in terms of the number of iterations required for an MCMC algorithm to reach a prescribed accuracy, under the assumption that each iteration possesses oracle access to evaluations of the forward map $\mathcal G$ (and potentially its derivatives). Accordingly, the polynomial-time guarantees in this section focus on iteration complexity rather than total arithmetic cost.

We briefly recall a standard Bayesian nonparametric formulation for the Darcy
model and related nonlinear inverse problems, as considered for instance in
\cite{NW22,nickl2023bayesian}. For $k\ge 1$, let $(e_k,\lambda_k)\in L^2(\mathcal O)\times (0,\infty)$ denote the $L^2(\mathcal O)$-orthonormal eigen-basis of the negative Dirichlet Laplacian on $\mathcal O$. Following \cite{NVW18,nickl2023bayesian,AW24}, we parameterize non-negative
diffusivities through truncated series expansions in this basis.
Specifically, for $\theta\in\R^D$ define
\[f_\theta := \Phi \circ F_\theta,~~~\text{with}~~~F_\theta (x) := \sum_{k=1}^D \theta_ke_k (x).\]
Here, $\Phi: \R\to (f_{min},\infty)$ is a smooth and strictly increasing `link function' satisfying $\Phi'>0$ and with bounded derivatives of all orders, $\sup_{x\in \R}|\Phi^{(k)}| <\infty$, as well as $\inf_{x\ge z}\Phi'(z)>0$ for all $z\in \R$. We place a Gaussian diagonal prior on $\theta$ of the form
\[\Pi =  N(0, N^{-\frac{d}{2\alpha+d}} \Sigma),\qquad \Sigma = \text{diag}(\lambda_j^{-\alpha}: 1\le j\le D)\in \R^{D\times D}.\]
Denoting the data by $Z^{(N)} = ((Y_i,X_i):1\le i\le N)$, the posterior density is, up to a multiplicative constant, given by
\begin{equation}\label{Darcy-post}
    \pi(\theta|Z^{(N)}) \propto \exp\Big(-\frac 12 \sum_{i=1}^N[Y_i-\mathcal G(f_\theta)(X_i)]^2)-\frac 12 \theta^T \Sigma^{-1}\theta\Big),~~~ \theta \in\R^D.
\end{equation}
Let $\theta_0\in \ell^2(\N),~\theta_{0}^{(D)}\in \R^D$ respectively denote the series coefficients of $F_0$ as well as their truncated counterpart;
\[ \theta_{0,k} = \langle F_0,e_k\rangle_{L^2},~~~~~~~\theta_{0}^{(D)} = (\theta_{0,k}:1\le k\le D). \]
Let $\tilde H^\alpha(\mathcal O)$ denote the spectrally defined Sobolev spaces arising from the Dirichlet Laplacian
\[ \tilde H^\alpha(\mathcal O) := \Big\{ f = \sum_{k=1}^\infty a_k e_k~~\text{for}~~\sum_{k=1}^\infty a_k^2 \lambda_k^{\alpha} <\infty \Big\}. \]
The following assumption encapsulates the  regularity assumptions required in Section 5.3 of \cite{nickl2023bayesian}, where the \textit{localized} sampling problem for the Darcy posterior distribution \eqref{Darcy-post} is studied using a `gradient stability' approach akin to \cite{NW22}.

\begin{ass}\label{polytime-MCMC-ass}
    Let $f_0= \Phi\circ F_0$ for some function $F_0 \in \tilde H^\alpha(\mathcal O)$ for some $\alpha \ge 22$, where $d\ge 3$ and $\mathcal O\subseteq \R^d$ is a bounded, smooth domain. For $N\ge 1$ let the dimension be chosen as $D=D_N:= a N^{\frac{d}{2\alpha+d}}$, for some $a>0$.
\end{ass}

The high smoothness requirement $\alpha\ge 22$ is inherited from
\cite{nickl2023bayesian} and is not intrinsic to our approach. In Theorem 5.3.5 of that reference, it is shown that under Assumption \ref{polytime-MCMC-ass}, the posterior density $\pi(\theta|Z^({N})$ satisfies the following two properties with probability tending to $1$ under the law of $Z^{(N)}$. Firstly, it is log-concave on a region of size $\{ \theta: \|\theta-\theta_{0,D}\|_{\R^D} \le r \}$ with $r=cD^{-8/d}$, with $c>0$ a small enough constant. Secondly, $\Pi(\cdot | Z^{(N)})$ concentrates the overwhelming majority of its mass in this region. Thus, a sufficient condition (asymptotically as $N\to\infty$) for successful warm-start initialisation is to find a point  $\vartheta_{init}\in \R^D$ in polynomial-time such that with probability tending to $1$ as $N\to\infty$, $\|\vartheta_{init} - \theta_{0,D}\|_{\R^D}\le D^{-8/d}/\log N.$

%Intuitively it is related to the `gradient stability' properties in the Darcy flow problems which determine the size of the region on which the posterior distribution is, with high probability, log-concave.
%it was proven that the posterior density $\pi(\theta|Z^{(N)})$ is locally log-concave on a region $\|\vartheta_{init} - \theta_{0,D}\|_{\R^D}\le D^{8/d}/\log D$ with high probability. Thus, an initialisation into this region will enable the use local log-concavity results to prove that 

\begin{thm}\label{darcy-initialisation}
    Consider the Darcy problem under Assumption \ref{polytime-MCMC-ass}. Then, there exists an estimator polynomial-time computable estimator $\vartheta_{init}=\vartheta_{init}(Z^{(N)})\in \R^D$ such that 
    $\P_{f_0}^N \big(\|\vartheta_{init} - \theta_{0,D}\|_{\R^D}\le D^{-8/d}/\log N\big)\to 1$ as $N\to\infty$.
\end{thm}

The proof of Theorem \ref{darcy-initialisation} is given in Section \ref{sec-proof-init}. It turns out that $\alpha \ge 12$ is in fact sufficient to achieve a convergence rate of $D^{-8/d}/\log N$. As a consequence of Theorem \ref{darcy-initialisation} together with \cite[Theorem 5.3.6]{nickl2023bayesian}, it follows that the high-dimensional posterior mean is computable via MCMC within polynomially many iterations; this is the subject of our next result. For `Gaussian innovations' $\xi_k\sim^{i.i.d.} N(0,I_{D\times D})$, we define the following time-discretized Langevin-type MCMC algorithm on $\R^D$ with step-size $\delta >0$,
\begin{equation}\label{MCMC-ULA}
    \begin{cases}
        \vartheta_0 &= \vartheta_{init},\\
        \vartheta_{k+1}&= \vartheta_k + \delta [ \nabla \tilde \ell_N (\vartheta_k) + \Sigma^{-1}\vartheta_k] + \sqrt{2\gamma} \xi_{k+1}, ~~~k\ge 0,
    \end{cases}
\end{equation}
where $\tilde \ell_N$ is the globally log-concave `surrogate likelihood' function constructed in \cite[eq.~(5.8)]{nickl2023bayesian} which coincides with the true log-likelihood in a neighborhood of
$\vartheta_{init}$, see also \cite{NW22} for a detailed construction. Denoting the posterior mean by 
\[ \bar \theta_N =\bar \theta_N(Z^{(N)}):= \E[\theta | Z^{(N)}] = \int_{\R^D} \theta d\pi(\theta|Z^{(N)}) \in \R^D, \]
we have the following theorem.

\begin{thm}\label{thm-darcy-MCMC}
    Suppose Assumption \ref{polytime-MCMC-ass} holds. Fix $P>0$. For any $N\ge 1$ let $D=D_N$ be chosen as above and let $\rho= \rho_N \ge N^{-P}$ denote the desired precision level. Then, there is some suitably chosen step size $\delta=\delta(N,D,\eps)$ and exponent $\gamma>0$ such that after polynomially many $J= O(N^{\gamma})$ iterations, the ergodic average
    \[  \bar \theta^{MCMC}_J = \frac 1{J} \sum_{k=J_{in}}^{J+J_{in}} \vartheta_k \in \R^D\]
    of the MCMC chain $(\vartheta_k:k \ge 0)$ from \eqref{MCMC-ULA} satisfies that
    \[ \P \big(  \|\bar \theta^{MCMC}_J - \E[\theta | Z^{(N)}]\|_{\R^{D}} \le \rho_N \big) \xrightarrow{N\to\infty}1. \]
    Here $\P$ denotes the product probability of the law of the data $P^N_0$ and the randomness of the MCMC chain $(\vartheta_k:k \ge 0)$.
\end{thm}

\subsection{Adaptation to unknown smoothness}\label{sec-adaptive}

We conclude this section by presenting an adaptive version of the plug-in M-estimator,
which adapts to the unknown smoothness of the ground truth.
As it turns out, the achievable convergence rates are governed by the smoothness of
$u_0=\mathcal G(f_0)$ rather than that of $f_0$ itself. Our next theorem is formulated in the infinite-dimensional setting; a corresponding
high-dimensional version could be developed analogously to the constructions in
Section~\ref{sec-high-dim}. For simplicity, and in order to remain compatible with the PDE theory employed
throughout, we restrict attention to integer-valued smoothness indices. We define the adaptive plug-in estimation procedure as follows. For integers $a,b\ge 0$, we write $[a:b]=\{a,a+1,\dots,b\}$.
\begin{enumerate}
    \item Form the preliminary estimator
    \begin{equation}\label{uhat-adaptive}
    (\hat \beta,\hat u_N)
    \in
    \argmin_{\beta \in [\beta_{\min}:\beta_{\max}],~u\in H^\beta(\mathcal O)}
    \frac{1}{N}\|Y-u(X)\|_{\R^N}^2
    + \mu_{N,\beta}^2\big(\|u\|_{H^\beta(\mathcal O)}^2 + A^2\big),
    \end{equation}
    for a suitable constant $A>0$, and with $\mu_{N,\beta}:= N^{-\frac{\beta}{2\beta+d}}$.
    \item With regularisation parameter $\nu_{N,\hat \beta}:= N^{-\frac{\hat\beta-2}{2\hat\beta+d}}$, compute the plug-in estimate
    \begin{equation}\label{fhat-adaptive}
    \hat f_N
    \in
    \argmin_{f\in H^{\alpha_{\min}}(\mathcal O)}
    \|L_f \hat u_N - g\|_{L^2(\mathcal O)}^2
    + \nu_{N,\hat\beta}^2 \|f\|_{H^{\alpha_{\min}}(\mathcal O)}^2.
    \end{equation}
\end{enumerate}
We emphasize that adaptation is achieved entirely through the preliminary estimator
$\hat u_N$ in Step~1, while the inverse step in \eqref{fhat-adaptive} is performed
over a fixed Sobolev space $H^{\alpha_{\min}}(\mathcal O)$, with adaptation entering
only through the data-driven choice of the regularization parameter $\nu_{N,\hat\beta}$.

\begin{thm}\label{thm-adaptive}
Let $(\hat u_N,\hat f_N)$ be the estimators defined in
\eqref{uhat-adaptive}–\eqref{fhat-adaptive}.
Let $\alpha_{\min}>d/2+1$ and suppose that $f_0\in H^{\alpha_{\min}}(\mathcal O)$.
Let $2\le \beta_{\min}<\beta_{\max}<\infty$ be integers, and define
\[
\beta_0 := \max\big\{ \beta \in [\beta_{\min}:\beta_{\max}] : u_0\in H^\beta(\mathcal O)\big\}.
\]
Then there exist constants $C_1,C_2>0$ such that
\begin{align}
\P_{f_0}^N(\hat\beta\ge\beta_0) &\xrightarrow{N\to\infty} 1, \label{alpha-consistency}\\
\P_{f_0}^N\big(\|\hat u_N-u_0\|_{L^2(\mathcal O)}
\ge C_1 N^{-\frac{\beta_0}{2\beta_0+d}}\big) &\xrightarrow{N\to\infty} 0, \label{adaptive-u-consistency}\\
\P_{f_0}^N\big(\|\hat f_N-f_0\|_{L^2(\mathcal O)}
\ge C_2 N^{-\frac{\beta_0-2}{2\beta_0+d}}\big) &\xrightarrow{N\to\infty} 0. \label{adaptive-f-consistency}
\end{align}
The constant $C_1$ may be chosen uniformly over classes
$\{f_0:\|\mathcal G(f_0)\|_{H^{\beta_0}(\mathcal O)}\le R\}$, while $C_2$ may be chosen
uniformly over
$\{f_0:\|\mathcal G(f_0)\|_{H^{\beta_0}(\mathcal O)}\le R,\ \|f_0\|_{C^1(\mathcal O)}\le R\}$.
\end{thm}

The proof of Theorem \ref{thm-adaptive} can be found in Section \ref{sec-adaptive-proof}. Theorem~\ref{thm-adaptive} shows that the procedure
\eqref{uhat-adaptive}–\eqref{fhat-adaptive} adapts to the unknown smoothness of
$u_0=\mathcal G(f_0)$ rather than that of $f_0$. The adaptive procedure is formulated in terms of nonparametric M-estimation,
primarily for ease of exposition.
Alternative adaptation strategies, such as Lepskii-type methods \cite{lepskii1991},
could potentially yield comparable results. An advantage of the present approach is that Step~1 simultaneously produces an estimate of the smoothness index, leading to adaptation with respect to Sobolev losses $H^\gamma(\mathcal O)$ for all $\gamma\in[0,\beta_0]$ in particular for the $H^2(\mathcal O)$-norm relevant in the generalized stability estimates. To the best of our knowledge, the only previous adaptive rates in the Darcy model
were those \cite{GN19}, where adaptation is achieved for the forward problem but not for the diffusivity $f_0$, see Theorems 8-10 in \cite{GN19}. Note that adaptation is possible here without incurring additional logarithmic factors in the convergence rate, due to the $L^2$-loss function employed here, cf.~Chapter 8 in \cite{GN16}.

\subsection{Results for the Schr\"odinger model}\label{sec-schroe}

We now turn to our second main PDE example, an inverse problem for the steady-state
Schr\"odinger equation studied, for instance, in \cite{N20,NW22,AW24}.
The main result of this section is Theorem~\ref{Schroed-rate}, which establishes
minimax-optimal convergence rates for both the PDE-penalized and plug-in M-estimators. The forward map $\mathcal G(f)$ is defined through the unique solution $u=u_f$ of
\begin{equation}\label{schroe}
	\begin{cases}
		\frac 12 \Delta u - f u = 0 & \textnormal{on } \mathcal O,\\
		u=g & \textnormal{on } \partial\mathcal O,
	\end{cases}
\end{equation}
where $g\in C^\infty(\partial\mathcal O)$ is a known strictly positive boundary datum.
Whenever $f\ge 0$ and $f\in C^\eta(\mathcal O)$ for some $\eta>0$, elliptic theory yields a unique
classical solution $u\in C^{2+\eta}(\mathcal O)\cap C(\overline{\mathcal O})$.
See \cite{CZ95} for a comprehensive PDE theory of Schr\"odinger-type equations and \cite{BU10}
for the occurrence of this problem in photoacoustic tomography (PAT).
For $f:\mathcal O\to\R$, write
\[
L_f u := \frac12 \Delta u - f u .
\]
We consider the following two estimation procedures, which are direct analogues of those
introduced for the Darcy model. Define, for $(u,f)\in H^{\alpha+2}(\mathcal O)\times H^\alpha(\mathcal O)$, the objective
\begin{equation}\label{eq:joint-criterion}
J(u,f)
:= \frac{1}{N}\sum_{i=1}^{N}\bigl(Y_i-u(X_i)\bigr)^2
 + \lambda_N^2 \bigl\|L_f u - g\bigr\|_{L^2(\mathcal O)}^2
 + \mu_N^2\Bigl(\|u\|_{H^{\alpha+2}(\mathcal O)}^2 + \|f\|_{H^\alpha(\mathcal O)}^2\Bigr).
\end{equation}
The PDE-penalized M-estimator is defined as any measurable minimizer
\begin{equation}\label{eq:joint-estimator}
(\hat u_N,\hat f_N)
\in \argmin_{(u,f)\in H^{\alpha+1}(\mathcal O)\times H^\alpha(\mathcal O)} J(u,f),
\end{equation}
with regularization parameters
\[
\lambda_N = N^{-\frac{2}{2(\alpha+2)+d}},
\qquad
\mu_N = N^{-\frac{\alpha+2}{2(\alpha+2)+d}}.
\]
Alternatively, the plug-in M-estimator is defined via the two-stage procedure
\begin{equation}\label{eq:plugin-u}
\hat u_N
= \argmin_{u\in H^{\alpha+1}(\mathcal O)}
\left\{
\frac{1}{N}\sum_{i=1}^{N}\bigl(Y_i-u(X_i)\bigr)^2
+ \mu_N^2\|u\|_{H^{\alpha+1}(\mathcal O)}^2
\right\},
\end{equation}
followed by
\begin{equation}\label{eq:plugin-f}
\hat f_N
= \argmin_{f\in H^{\alpha}(\mathcal O)}
\left\{
\bigl\|L_f\hat u_N - g\bigr\|_{L^2(\mathcal O)}^2
+ \nu_N^2\|f\|_{H^{\alpha}(\mathcal O)}^2
\right\},
\end{equation}
where $\mu_N = N^{-\frac{\alpha+2}{2(\alpha+2)+d}}$, $\nu_N = N^{-\frac{\alpha}{2(\alpha+2)+d}}$.

To avoid excessive technicalities, we derive the following convergence result in high probability rather than mean squared error.

\begin{thm}\label{Schroed-rate}
Let $d\ge 1$ and suppose $\mathcal O\subseteq \R^d$ is a smooth, bounded domain.
Let $\alpha> d/2+1$ and $\alpha\ge 3$, and fix constants $f_{\min}\ge 0$ and $R>0$.
Suppose that the data \eqref{PDE-data} arise from some
$f_0\in \mathcal F(\alpha,f_{\min},R)$, and let $(\hat u_N,\hat f_N)$ denote
either the PDE-penalized M-estimator defined in
\eqref{eq:joint-estimator} or the plug-in M-estimator defined in
\eqref{eq:plugin-u}–\eqref{eq:plugin-f}.
There exists $C>0$ depending only on $\mathcal O,d,\alpha,f_{\min},R$ such that
\begin{align}
\P_{f_0}^N\Big(
\|\hat u_N - \mathcal G(f_0)\|_{L^2(\mathcal O)}^2
\ge C N^{-\frac{2(\alpha+2)}{2(\alpha+2)+d}}
\Big) &\xrightarrow{N\to\infty} 0, \label{Schroed-forward}\\
\P_{f_0}^N\Big(
\|\hat f_N - f_0\|_{L^2(\mathcal O)}^2
\ge C N^{-\frac{2\alpha}{2(\alpha+2)+d}}
\Big) &\xrightarrow{N\to\infty} 0.
\end{align}
\end{thm}

The proof is given in Section \ref{sec-schroed-proof}. Like for the Darcy model, a key ingredient to the proof is that the Schr\"odinger problem satisfies the following generalized stability estimate.

\begin{restatable}{lem}{schroedgeneralizedstability}
\label{schroed-generalized-stability}
Let $\mathcal O\subseteq \R^d$ be a smooth, bounded domain and let $c_{\min}>0$.
Suppose that for $i\in\{1,2\}$ the functions
$u_i\in H^2(\mathcal O)$, $f_i\in C(\mathcal O)$ and $h_i\in L^2(\mathcal O)$ satisfy
\[
L_{f_1}u_1=h_1
\qquad\text{and}\qquad
L_{f_2}u_2=h_2
\quad\text{in }\mathcal O.
\]
Assume moreover that $(u_1,f_1,h_1)$ satisfies $u_1\in C^2(\mathcal O)$ and $h_1\in L^\infty(\mathcal O)$, and that
\[
\inf_{x\in\mathcal O,\, i=1,2} u_i(x) \ge c_{\min}.
\]
Then there exists a constant $C>0$, depending only on $c_{\min}$ and upper bounds on
$\|u_1\|_{C^2(\mathcal O)}$ and $\|h_1\|_{L^\infty(\mathcal O)}$, such that
\[
\|f_1-f_2\|_{L^2(\mathcal O)}
\le
C\big(\|u_1-u_2\|_{H^2(\mathcal O)} + \|h_1-h_2\|_{L^2(\mathcal O)}\big).
\]
\end{restatable}

Using the same techniques as in the case of Darcy flow, one could also derive analogous results for a high-dimensional discretization of the Schr\"odinger problem computable in sub-quadratic time, as well as results on adaptive estimation. Since these results would closely parallel the results in Sections \ref{sec-high-dim}-\ref{sec-adaptive}, we do not pursue this extension here.

\section{An abstract theorem for generalized M-estimators}\label{sec:general-M}

We now formulate a general result on convergence rates of `augmented M-estimators' building on classical ideas developed in \cite{V00}, used throughout the proofs. The formulation is intentionally
slightly technical in order to allow maximum flexibility in its application to the
various estimation procedures considered in this paper.

Let $d\ge 1$ and let $\mathcal O\subseteq \R^d$ be a Borel measurable, non-empty subset. Let $X_i\in \mathcal O$ be the design points at which we are given regression measurements
\[ Y_i = u_0(X_i) +\eps_i ,~~~~i= 1,...,N, \]
where $u_0:\mathcal O\to \R$ is some regression function and $\eps_i\in \R$ are independent Gaussian measurement errors with variance at most $\sigma^2>0$,
\[ \eps_i \sim N(0,\sigma_i^2),~~~ \sup_{1\le i\le N} \sigma_i \le \sigma. \]
The result in this section will not make any assumption regarding whether the $X_i$ are deterministic or random. %At the expense of additional technicalities, one could generalize our findings to the case of general subgaussian noise distributions, but we do not pursue this further here.

We consider M-estimators in two variables $\eta\in H$ and $\theta \in\Theta$, where $\eta\in H$ should be interpreted as the variable which parameterizes the regression functions, and $\theta\in \Theta$ is another statistical parameter of interest. Throughout we fix some parameterization of regression functions on $\mathcal O$, 
\[ \eta \mapsto u_\eta~~\text{for all}~~\eta\in H,~~u_\eta: \mathcal O\to \R. \]
We will consider the minimization of functionals of the form
\[ J: H\times \Theta \to \R, ~~~~J(\eta, \theta) = \frac 1N \sum_{i=1}^n (Y_i- u_\eta(X_i))^2 + \mathcal R(\eta,\theta), \]
where $\mathcal R: H\times \Theta \to [0,\infty)$ is some non-negative penalty term, which be arbitrary in the context of this section. Throughout, let us assume that minimizers $(\hat \eta, \hat \theta)$ of $J$ exist, and that there exists a choice of $(\hat \eta, \hat \theta)$ which is a measurable function of the data $(Y_i,X_i)_{i=1}^N$, which in the settings relevant to our theorems above, one may verify using well-known measurable selection theorems, e.g.~\cite{brown1973measurable}. We denote the empirical norm by
\[ \frac 1N \sum_{i=1}^N h(X_i)^2 = \|h\|_{N}^2,~~~h:\mathcal O\to \R. \]
%Note that for a.e.~ defined functions $h\in L^2(\mathcal O)$, $\|h\|_{N}$ is still a well-defined random variable for $X_i\sim \text{Unif}(\mathcal O)$.
For any $(\eta,\theta)\in H\times \Theta$ and function $u:\mathcal O\to \R$, we define the following `discrepancy functional'
\[\tau^2((\eta,\theta), u):=\|u_\eta-u\|^2_{N}+\mathcal R(\eta,\theta). \]
Using this discrepancy functional, we introduce the following localized classes of regression functions. For any `centering' element $\eta^*:\mathcal O\to \R$ and $R>0$, let
\[\mathcal U(\eta^*,R) :=\big\{ u= u_\eta~~\text{for some}~~(\eta,\theta)\in H\times \Theta~~\text{with}~~\tau^2((\eta, \theta),u_{\eta^*})\le R^2\big \}.\]
Intuitively, $\mathcal U(\eta^*,R)$ is the set of all the regression functions $u_\eta$ represented by parameter pairs $(\eta,\theta)$ which are `close to' the reference point $\eta^*$, as measured by the $\tau^2$ functional. We may think of $\eta^*$ as the `best approximation' of the ground truth. The metric entropy integral of $\mathcal U(\eta^*,R)$ with respect to the empirical norm is denoted by
\begin{align*}
\mathcal I(\eta^*,R) :=&\sigma R+ \int_{0}^{\sigma R}H^{1/2}(\rho,\mathcal U(\eta^*,R),\sigma\|\cdot\|_{N})d\rho.
\end{align*}

\begin{thm}\label{thm-random-design} For $(\eta^*,\theta^*) \in H\times \Theta$ an arbitrary element, suppose that
\begin{align}\label{entropy-ineq}
    \Psi(R)\ge \mathcal I(\eta^*,R)
\end{align}is an upper bound such that $R\mapsto \Psi(R)/R^2$ is non-increasing. There exists a universal constant $C>0$ such that for all such $(\eta^*,\theta^*)$ and $\Psi(\cdot)$, any  $N\ge 1,\delta_N>0$ with 
\[ \sqrt N \delta_N^2 \ge C\Psi(\delta_N), \]
and any $R\ge \delta_N$, it holds that
	\begin{equation}\label{conc-random-design}
	\P_{0}^N   \big( \tau^2((\hat \eta,\hat \theta),u_0)\ge 2\big(\tau^2 (( \eta^*,\theta^*),u_0)+R^2)  \big) \le C\exp(-NR^2/C).
	\end{equation}
    Moreover, we have the following moment bound for all $k\ge 1$,
    \begin{equation}\label{MSE-random-design}
        \E_{0}^N \big[ \tau^{2k}((\hat \eta,\hat \theta),u_0) \big] \le C_k \Big(
\tau^{2k}((\eta^*,\theta^*),u_0)
+ \delta_N^{2k}
+ \frac{\delta_N^{2k-2}\sigma^2}{N}\Big),~~~\text{some}~C_k>0.
    \end{equation}
\end{thm}

The proof, presented in Section \ref{appendix-M-estimationproofs}, generalizes classical empirical process arguments for penalized least squares estimators, in particular~\cite{V00}, to the present augmented parameter setting. The key idea is a localization argument based on the discrepancy functional $\tau^2$,
combined with concentration inequalities for Gaussian processes indexed by the localized
class $\mathcal U(\eta^*,R)$. 
The monotonicity assumption on $\Psi(R)/R^2$ then allows one to use a standard `peeling argument'.

\begin{rem}\normalfont
At the expense of additional technicalities, our results could be extended to the case of
independent sub-Gaussian noise variables, satisfying for some $\sigma>0$
\begin{equation}\label{sub-gauss}
    \E\big[e^{\lambda \varepsilon_i}\big] \le \exp(\lambda^2\sigma^2/2),
    \qquad \forall \lambda\in\R,\ \forall i\ge 1.
\end{equation}
In this case, the Borell--Sudakov--Tsirelson inequality used in the proof can be replaced
by a corresponding tail inequality for sub-Gaussian processes; see, for instance,
Corollary~8.3 in~\cite{V00}. We do not pursue this extension here.
\end{rem}

\begin{rem}[White Noise]\normalfont
A concentration result analogous to Theorem \ref{thm-random-design} can also be formulated in a white noise observation model, which is also commonly considered in the literature, see, e.g., \cite{NVW18,KSV24}. Here, $u_0$ may belong to a general separable Hilbert space $\H$ and one observes a single realization of the Gaussian process
\[
Y_N = u_0 + N^{-1/2}\mathbb W,
\]
where $\W$ is an isonormal centered white noise process indexed by $\H$, with covariance $\E[\W(g)\W(h)]=\langle g,h\rangle_{\H}$. In this case, one considers a parameterization $\eta \mapsto u_\eta\in \H$, and the objective function is given by 
\[ (\hat \eta, \hat \theta) \in \argmin_{(\eta,\theta)\in H\times \Theta}  -2\langle Y_N,u_\eta\rangle + \|u_\eta\|_{\mathbb H}^2 + \mathcal R(\eta,\theta). \]
The proof proceeds along essentially the same lines as in the discrete-design case, with the $\|\cdot\|$-norm replacing the empirical norm $\|\cdot\|_N$ throughout. Since the applications in this paper
concern the discretized case of finitely many $N\ge 1$ measurements---for which the question of computational complexity is most natural---we do not pursue the white noise formulation further here.
\end{rem}

\section{Further remarks and future directions}\label{sec-further-dis}

\subsubsection*{Boundary and other penalties}
Our objective does not include an explicit penalty enforcing boundary conditions. 
In physics-informed neural networks (PINNs), by contrast, boundary mismatch terms are often incorporated in the loss function to promote adherence to the PDE constraint \cite{PINN}. 
For the PDE-based inverse problems studied in this paper, this turns out not to be necessary: the generalized stability estimates established below show that interior information alone already controls the parameter error, even when the PDE constraint is only enforced weakly. 
That said, in other inverse problems or for different operator equations, additional penalties---including boundary terms or other structural regularizers---may be useful, and our framework readily allows such extensions.

\subsection*{Uncertainty quantification (UQ)} 

A natural question is whether the present non-likelihood-based methods can deliver not only optimal convergence rates, but also finer notions of optimality relevant for statistical uncertainty quantification (UQ), such as semiparametric efficiency in the Cram\'er--Rao sense \cite{V98}. Concretely, one may ask whether estimators based on our surrogate objectives can achieve the smallest possible asymptotic variance for regular linear functionals of the form
\[
\sqrt{N}\,\langle \hat f_N - f_0, \psi\rangle,
\]
for suitable test functions $\psi$. In settings where a Bernstein--von Mises theorem holds, Bayesian procedures are known to be asymptotically optimal for such functionals. In the Darcy inverse problem, however, $\sqrt{N}$-consistent estimation is not attainable along arbitrary smooth directions $\psi$, due to range restrictions induced by the forward map $\mathcal G$; see \cite{nickl2023bayesian} and Chapter~25 of \cite{V98}. The objective functions studied here are based on tractable surrogates of the likelihood and are therefore no longer directly linked to the Fisher information operator $D\mathcal G_{f_0}^\ast D\mathcal G_{f_0}$. As a consequence, while they yield optimal rates, they may in principle lead to a suboptimal asymptotic covariance structure. A detailed analysis of asymptotic distributions and efficiency---including limiting laws for $\hat u_N$ and $\hat f_N$---is beyond the scope of this paper and constitutes an interesting direction for future work. We also refer to \cite{S23} for related results on limiting distributions of M-estimators.

%\section{Numerical illustration}

%\red{FNO paper also studies Darcy. There $f=1$ constant source is used, and a 421x421 grid of the unit square with Dirichlet conditions.}

%\paragraph{Geometric interpretation}
%\pink{Likelihood-based methods have shown to possess statistically (and sometimes minimax optimal) convergence properties in non-linear inverse problems, including MAP estimators \cite{NVW18} and Bayes methods cite{Line of papers by Nickl etc.} However, these methods suffer from the non-linearity of $\mathcal G$ and therefore, it is yet unknown whether they can be computed efficiently.}

%% if your bibliography is in bibtex format, uncomment commands:
\bibliographystyle{imsart-number} % Style BST file (imsart-number.bst or imsart-nameyear.bst)
\bibliography{bib_polytime, ercbib}       % Bibliography file (usually '*.bib')

\newpage \appendix

\section{Auxiliary PDE results}

We present two key analytical properties for our main PDE examples, the Darcy model and the Schr\"odinger model. The properties needed are: firstly, `boundedness' properties for the forward map $\mathcal G$ between suitable (Sobolev) smoothness balls and secondly, \textit{generalized stability estimates} for $\mathcal G^{-1}$ discussed above. Note that the boundedness requirement for the forward map $\mathcal G$ is weaker than what is required in convergence proofs for Bayes and MAP estimators \cite{NVW18,GN19,nickl2023bayesian}, where local Lipschitz estimates for $\mathcal G$ are required. The technical reason why we do not require such Lipschitz estimates is that the regularity of $u$ is explicitly penalized.

\subsection{Darcy flow}

The following lemma asserts that the Darcy operator maps $H^\alpha$-bounded sets to $H^{\alpha+1}$-bounded sets.

\begin{lem}\label{lem-darcy-boundedness}
    Consider the Darcy operator $\mathcal G$ from (\ref{darcy}), with $g\in C^\infty(\mathcal O)$. Let $\mathcal F=\mathcal F (\alpha, f_{min}, R)$ be given by \eqref{F-def} for some $f_{min}>0$, $R>0$ and $\alpha > d/2+1$. Moreover, suppose that $g\in H^{\alpha-1}(\mathcal O)$. There exists a constant $C=C(d,\mathcal O,f_{min},g)$ such that for all $f\in H^\alpha(\mathcal O)$ with $f\ge f_{min}$ that 
    \[ \|\mathcal G(f)\|_{H^{\alpha+1}}\le C(1+\|f\|_{H^\alpha}^\gamma).\]
    In particular, it holds that 
    \[ \sup_{f\in\mathcal F (\alpha, f_{min}, R)} \|\mathcal G(f)\|_{H^{\alpha+1}} \vee \|\mathcal G(f)\|_{C^2} =: M (\alpha,f_{min},R) <\infty. \]
\end{lem}
\begin{proof}
    
The first statement follows from Lemma 23 of \cite{NVW18}, noting that the regularity assumption of $\alpha >d/2+2$ (inherited from Lemma 22 in \cite{NVW18}) can be weakened to $\alpha>d/2+1$ at no cost. It is shown in Lemma 23 of \cite{NVW18} that $\gamma>0$ can be chosen as $\gamma=\alpha^2+\alpha$. The second statement follows from the Sobolev embedding $C^2\subseteq H^{\alpha+1}$, applicable due to the assumption $\alpha>d/2+1$.
\end{proof}

Next, we turn to the proof of the generalized stability property for the Darcy model, which we restate for the convenience of the reader.

\darcygeneralizedstability*

\begin{proof}[Proof of Lemma \ref{darcy-generalized-stability}]
We divide the proof into three steps. Our proof builds on the approach of \cite[Lemma~24]{NVW18} using elliptic PDE theory.

\paragraph*{Step 1} For $f_1,~f_2$ as in the hypotheses, write $h=f_1-f_2$. By assumption, $h$ satisfies
\begin{align*}
    \nabla \cdot (h\nabla u_1) &= \nabla\cdot (f_1\nabla u_1)- \nabla \cdot (f_2\nabla u_2)-\nabla \cdot (f_2\nabla (u_1-u_2)) \\ &= \nabla\cdot (f_2\nabla (u_2-u_1 )) + (g_1-g_2).     
\end{align*}

The $L^2$-norm of the right hand side can further be upper bounded by
\begin{equation}\label{h-PDE}
    \| \nabla\cdot (f_2\nabla (u_1-u_2 ))\|_{L^2} + \|g_1-g_2\|_{L^2}\le 2\|f_2\|_{C^1}\|u_2-u_1\|_{H^2} + \|g_1-g_2\|_{L^2}.
\end{equation}

\paragraph*{Step 2: Multiplier identity via Green's formula} The overall goal is to lower bound the first term of the left hand side of (\ref{h-PDE}). To this end, we prove the following identity: For $v=e^{-\lambda u_1}h$ (where $\lambda>0$ is chosen later, for now arbitrary),
\begin{equation}\label{stab-identity}
    \langle \nabla \cdot (h\nabla u_1), he^{-\lambda u_1} \rangle_{L^2} = \langle \frac 12 \Delta u_1 + \lambda \|\nabla u_1\|^2, v^2\rangle_{L^2} + \frac 12 \int_{\partial \mathcal O} v^2\frac{\partial u_1}{\partial n},
\end{equation}
where $n=n(x)$ denotes the outward normal vector on $x\in \partial \mathcal O$. Note that since $\|u_1\|_{C^2(\mathcal O)}\le B$ and $f_1-f_2\in C^1(\mathcal O)$ (with finite $C^1(\O)$-norm), the functions $v$, $u_1$ and $\nabla u_1$ are Lipschitz on $\mathcal O$ and admit a unique continuous extension to $\overline{\mathcal O}$; the boundary term is understood in this sense. To see \eqref{stab-identity}, we use Green's identity (integration by parts) to obtain
\begin{align*}
    \langle\Delta u_1,v^2\rangle + \langle \nabla u_1,\nabla (v^2)\rangle =\int_{\mathcal O} \nabla \cdot (v^2\nabla u_1)= \int_{\partial \mathcal O} \frac{\partial u_1}{\partial n} v^2,
\end{align*}
and then also
\[ \langle\Delta u_1,v^2\rangle + \frac 12 \langle \nabla u_1,\nabla (v^2)\rangle = \frac 12 \langle\Delta u_1,v^2\rangle + \frac 12\int_{\partial \mathcal O} \frac{\partial u_1}{\partial n} v^2. \]
Moreover, by the definition $v=e^{-\lambda u_1}h$, we have
\[
\nabla v = e^{-\lambda u_1}\nabla h - \lambda e^{-\lambda u_1}h\nabla u_1
= e^{-\lambda u_1}(\nabla h - \lambda h\nabla u_1).
\]
Since $\nabla(v^2)=2v\nabla v$, we obtain
\begin{align*}
\frac12 \langle \nabla u_1, \nabla(v^2)\rangle_{L^2}
&= \langle \nabla u_1, v\nabla v\rangle_{L^2} \\
&= \Big\langle \nabla u_1,\; v e^{-\lambda u_1}(\nabla h - \lambda h\nabla u_1)\Big\rangle_{L^2} \\
&= \langle \nabla h \cdot \nabla u_1,\; v e^{-\lambda u_1}\rangle_{L^2}
    - \lambda \langle \|\nabla u_1\|^2,\; v^2\rangle_{L^2}.
\end{align*}
Using the previous equations and $he^{-2\lambda u_1 }= ve^{-\lambda u_1}$, it holds that
\begin{align*}
    \langle \nabla\cdot(h\nabla u_1),he^{-2\lambda u_1}\rangle_{L^2} &= \langle h \Delta u_1+\nabla h\cdot \nabla u_1, he^{-2\lambda u_1}\rangle_{L^2}\\
    &= \langle h \Delta u_1, he^{-2\lambda u_1}\rangle_{L^2} +\frac 12 \langle \nabla u_1,\nabla (v^2)\rangle_{L^2} + \langle \lambda\|\nabla u_1\|^2,v^2 \rangle_{L^2}\\
    &= \langle \Delta u_1, v^2\rangle_{L^2} + \frac 12 \langle \nabla u_1,\nabla (v^2)\rangle_{L^2} + \langle \lambda\|\nabla u_1\|^2,v^2 \rangle_{L^2}\\
    &= \langle \frac 12 \Delta u_1 + \lambda \|\nabla u_1\|^2, v^2\rangle_{L^2} + \frac 12 \int_{\partial \mathcal O} v^2\frac{\partial u_1}{\partial n}.
\end{align*}

\paragraph*{Step 3: Lower bound for \eqref{stab-identity}} We show that the right-hand side of \eqref{stab-identity} is bounded from below by a positive multiple of $\|h\|_{L^2(\mathcal O)}^2$, for a suitable choice of $\lambda=\lambda(g_{\min},B)>0$ in the definition of $v$.

\smallskip

\emph{(i) Positivity of the boundary term.}
Since $g_1\ge g_{\min}>0$ and $u_1|_{\partial\mathcal O}=0$, the strong maximum principle implies that
\begin{equation}\label{u1-negative}
    u_1(x)<0,\qquad x\in \mathcal O.
\end{equation}
Hence, by the Hopf boundary point lemma (see, e.g., \cite[Lemma~3.4]{GT98}), we have
\begin{equation}\label{Hopf-derivative}
    \frac{\partial u_1}{\partial n}(x) >0, \qquad x\in \partial\mathcal O,
\end{equation}
and consequently,
\begin{equation}\label{boundary-nonneg}
    \frac12\int_{\partial\mathcal O} v^2 \frac{\partial u_1}{\partial n}\,\ge 0.
\end{equation}

\smallskip

\emph{(ii) Pointwise lower bound on the multiplier.}
Recall that $u_1$ solves $\nabla\cdot(f_1\nabla u_1)=g_1$, hence
\begin{equation}\label{Lf-expansion}
    g_1(x) = f_1(x)\Delta u_1(x) + \nabla f_1(x)\cdot \nabla u_1(x),\qquad x\in \mathcal O.
\end{equation}
Using $f_1\ge f_{\min}$ and $\|\nabla f_1\|_{\infty}\le \|f_1\|_{C^1}\le B$, we infer that for every $x\in\mathcal O$,
\[
g_{\min}
\le f_1(x)\Delta u_1(x) + \|\nabla f_1\|_{\infty}\|\nabla u_1(x)\|
\le B\Delta u_1(x) + B\|\nabla u_1(x)\|.
\]
Hence, for every $x\in\mathcal O$, at least one of the following two inequalities holds:
\begin{equation}\label{either-or}
    \Delta u_1(x) \ge \frac{g_{\min}}{2B}
    \qquad\text{or}\qquad
    \|\nabla u_1(x)\| \ge \frac{g_{\min}}{2B}.
\end{equation}
Since $\|u_1\|_{C^2}\le B$, we have $\|\Delta u_1\|_{\infty}\le c(B)$ for some $c(B)<\infty$. Therefore, choosing $\lambda=\lambda(g_{\min},B)>0$ sufficiently large, we obtain the pointwise lower bound
\begin{equation}\label{bulk-lower}
    \frac12 \Delta u_1(x) + \lambda\|\nabla u_1(x)\|^2 \ge c_0,
    \qquad x\in \mathcal O,
\end{equation}
for some constant $c_0=c_0(g_{\min},B)>0$.

\smallskip

\emph{(iii) Conclusion.}
Combining \eqref{stab-identity}, \eqref{boundary-nonneg} and \eqref{bulk-lower} yields
\[
\big\langle \nabla\cdot(h\nabla u_1),\,he^{-2\lambda u_1}\big\rangle_{L^2}
\;\ge\; c_0 \|v\|_{L^2(\mathcal O)}^2.
\]
On the other hand, by Cauchy--Schwarz and the bound $\|e^{-2\lambda u_1}\|_{\infty}\le \exp(2\lambda B)<\infty$,
\[
\big|\big\langle \nabla\cdot(h\nabla u_1),\,he^{-2\lambda u_1}\big\rangle_{L^2}\big|
\le \|\nabla\cdot(h\nabla u_1)\|_{L^2}\,\|he^{-2\lambda u_1}\|_{L^2}
\le \exp(2\lambda B)\|\nabla\cdot(h\nabla u_1)\|_{L^2}\,\|h\|_{L^2}.
\]
Therefore,
\begin{equation}\label{CS-final}
    \|\nabla\cdot(h\nabla u_1)\|_{L^2}\,\|h\|_{L^2}
    \;\ge\; c\,\|v\|_{L^2(\mathcal O)}^2,
\end{equation}
for some $c=c(g_{\min},f_{\min},B,\mathcal O,d)>0$. Since $u_1$ is bounded and $v=e^{-\lambda u_1}h$, we have $\|v\|_{L^2}\asymp \|h\|_{L^2}$ with constants depending only on $\lambda$ and $B$. Inserting this into \eqref{CS-final} gives
\[
\|\nabla\cdot(h\nabla u_1)\|_{L^2}\;\gtrsim\; \|h\|_{L^2},
\]
and the claimed stability estimate now follows by combining this bound with \eqref{h-PDE}.
\end{proof}

\subsection{Schr\"odinger model}

The following Lemma asserts that the Schr\"odinger parameter-to-solution map maps $H^\alpha$-bounded sets to $H^{\alpha+2}$-bounded sets. It is a direct restatement of \cite[Lemma~27]{NVW18} (with the same assumptions on the boundary data $g$); we include it only for convenience. We also note that the second statement follows from the Sobolev embedding $H^{\alpha+2}\subseteq C^2(\mathcal O)$ for $\alpha>d/2$.

\begin{lem}\label{lem-schroedinger-boundedness}
    Consider the Schr\"odinger operator $\mathcal G$ from (\ref{schroe}), with fixed $g\in C^\infty(\partial \mathcal O)$. Then, for any $\alpha >d/2$ there exists a constant $C>0$ such that for all $f\in H^\alpha(\mathcal O)$ with $f\ge 0$
    \[ \|\mathcal G(f)\|_{H^{\alpha+2}}\le C(1+\|f\|_{H^\alpha}^\gamma).\]
    In particular, it holds that for any $R> 0$, 
    \[ \sup_{f\ge 0,~\|f\|_{H^\alpha}\le R} \|\mathcal G(f)\|_{H^{\alpha+2}} \vee \|\mathcal G(f)\|_{C^2} =: M (\alpha,R) <\infty. \]
\end{lem}

Next, we give a proof of the generalized stability property of the Schr\"odinger equation, which is restated for the convenience of the reader.

\schroedgeneralizedstability*

\begin{proof}[Proof of Lemma \ref{schroed-generalized-stability}]
    For $i=1,2$, it holds that $f_i = (\Delta u_i/2-h_i)/u_i$. The map $z\mapsto z^{-1}$ has Lipschitz constant bounded by $c_{min}^{-2}$ on the interval $[c_{min},\infty)$. Thus, the desired stability estimate is derived as follows:
    \begin{align*}
        \| f_1& -f_2\|_{L^2} = \big\| \frac{\Delta u_1/2-h_1}{u_1}-\frac{\Delta u_2/2-h_2}{u_2}\big\|_{L^2}\\
        &\le\big\| \frac{\Delta u_1-\Delta u_2}{2u_2} \big\|_{L^2} + \big\| \Delta u_1\big( \frac{1}{2u_1} - \frac{1}{2u_2}\big) \big\|_{L^2} +\big\| \frac{h_2-h_1}{u_2} \big\|_{L^2}+\big\| h_1\big(\frac{1}{u_1} - \frac{1}{u_2}\big)  \big\|_{L^2}\\
        &\le c_{min}^{-1} \|u_1-u_2\|_{H^2} + c_{min}^{-2} \|u_1\|_{C^2} \|u_1-u_2\|_{L^2}+ c_{min}^{-1}\big\| h_2-h_1 \big\|_{L^2}\\
        &\qquad\qquad + c_{min}^{-2}\|h_1\|_\infty \|u_1-u_2\|_{L^2}.\\
        &\le C(\|u_1-u_2\|_{H^2} + \|h_1-h_2\|_{L^2}),
    \end{align*}
    where the constant $C<\infty$ only depends on $c_{min}$, $\|u_1\|_{C^2}$ and $\|h_1\|_\infty$.
\end{proof}

\section{Proofs for Section \ref{sec:pde-results}}\label{sec-darcy-infdim-proof}

We present the proofs for the main results in Section \ref{sec:pde-results}. Throughout, the empirical norm will be denoted by $\|h\|_N:= \big(N^{-1}\sum_{i=1}^Nh^2(X_i)\big)^{1/2}$.

\subsection{Proof of Theorem \ref{Darcy-rate}}

    \subsubsection{PDE-penalized M-estimator}
    We first treat the case of the PDE-penalized M-estimator.
    
    \paragraph*{Step 1: A concentration inequality for the regularized empirical risk} We apply Theorem \ref{thm-random-design} with $\sigma=1$, $u\in H^{\alpha+1}(\mathcal O)$ in place of $\eta \in H$ (with the `identity parameterization') and $f\in H^\alpha(\mathcal O)$ in place of $\theta \in\Theta$, and with regularizer
    \begin{equation}\label{R-Darcy}
        \mathcal R(u,f)= \lambda_N^2 \|L_fu- g\|^2_{L^2(\mathcal O)} +  \mu_N^2 \Big(\|u\|_{H^{\alpha+1}(\mathcal O)}^2 + \|f\|_{H^{\alpha}(\mathcal O)}^2\Big),
    \end{equation}
    where we recall that $\lambda_N= N^{-\frac{2}{2(\alpha+1)+d}}$ and $\mu_N= N^{-\frac{\alpha+1}{2(\alpha+1)+d}}$. 
    Since $f_0\in \mathcal F$, by Lemma \ref{lem-darcy-boundedness} it holds that $u_0=\mathcal G(f_0)$ has finite $H^{\alpha+1}(\mathcal O)$-norm, 
    \[\sup_{f\in\mathcal F(\alpha,f_{min},R)} \|\mathcal G(f)\|_{H^{\alpha+1}(\mathcal O)}=:M <\infty. \]
    Set $(f^*,u^*)=(f_0,u_0)$. The localized regression function classes then satisfy, by the definition of $\mathcal R$, that
    \[ \mathcal U(\eta^*,R) \subseteq \{u: \|u\|_{H^{\alpha+1}}\le R/\mu_N\}. \]
    Consequently, using standard entropy bounds for Sobolev classes on bounded and smooth domains---see Theorem 4.10.3 in~\cite{Triebel1978}---we can upper bound the entropy integral in the empirical norm by
    \begin{equation}\label{entropy-integral-bound}
        \begin{split}
        \int_0^R\sqrt{H(\mathcal U(\eta^*,R),\|\cdot\|_\infty,\rho)} d\rho &\lesssim \int_0^R \Big(\frac{R}{\mu_N \rho}\Big)^{\frac{d}{2(\alpha+1)}}d\rho \lesssim R \mu_N^{-\frac{d}{2(\alpha+1)}}.
        \end{split}
    \end{equation}
    Here we have used that $\alpha+1>d/2$ by assumption. Thus, for an appropriate constant $C_1>1$, we can set
    \begin{equation}\label{Psi-definition}
        \Psi(R) := C_1 R\mu_N^{-\frac{d}{2(\alpha+1)}} \ge R+\int_0^R\sqrt{H(\mathcal U(\eta^*,R),\|\cdot\|_\infty,\rho)} d\rho,
    \end{equation}
    where we used that $\mu_N\le 1$ for all $N\ge 1$. Note that $\Psi(R)/R^2$ is decreasing in $R>0$. By our choice of $\mu_N$, it holds that
$\mu_N^{-d/(2(\alpha+1))} = N^{d/(4(\alpha+1)+2d)}$.
Then, for $C>0$ being the constant in Theorem \ref{thm-random-design}, we can verify that the condition 
    \[\sqrt N \delta_N^2 \ge C\Psi(\delta_N)=C C_1 \delta_N N^{\frac{d}{4(\alpha+1)+2d}} \] holds for $\delta_N =C_2N^{-\frac{\alpha+1}{2(\alpha+1)+d}}$ with $C_2>0$ chosen sufficiently large. Let
    \[\tau^2((u,f),u_0) = \|u-u_0\|_{N}^2 +\mathcal R(u,f).\]
    It holds that
    \[ \tau^2((u^*,f^*),u_0)=\mu_N^2 (\|f_0\|_{H^\alpha(\mathcal O)}^2+\|u_0\|_{H^{\alpha+1}(\mathcal O)}^2 ) \le (R^2+M^2) \mu_N^2\lesssim N^{-\frac{\alpha+1}{2(\alpha+1)+d}}, \]    
    Theorem \ref{thm-random-design} thus yields that for all $L\ge L_0>0$ ($L_0$ large enough) and some $C_3>0$,
    \begin{equation}\label{empirical-rate}
        \P_0^N \Big(\tau^2((\hat u_N,\hat f_N),u_0) \ge L^2N^{-\frac{2(\alpha+1)}{2(\alpha+1)+d}} \Big)\le C_3\exp\big(-L^2N^{\frac{d}{2(\alpha+1)+d}}/C_3\big),
    \end{equation}
    as well as the moment bound
    \[ \E_0^N \big[\tau^{2k}((\hat u_N,\hat f_N),u_0)\big] \lesssim  \delta_N^{2k},~~~\text{any}~k\ge 1. \]
    
\vspace{.3cm}
    
     \paragraph*{Step 2: Convergence of the $L^2$-risk} To complete the proof of \eqref{Darcy-forward}, we need to show that we can use the $L^2$-norm in place of the empirical norm. Define the $L^2$-counterpart of the $\tau$-functional by
    \[\mathcal T^2((u,f),u_0) = \|u-u_0\|_{L^2}^2+\mathcal R(u,f).\]
    Take $K \ge L_0^2$ and $K\ge 16$ large enough, and consider the event
    \[A_N = \big\{ \mathcal T^2((\hat u_N,\hat f_N),u_0) \ge K^2 \mu_N^2,~\tau ^2((\hat u_N,\hat f_N),u_0)\le K\mu_N^2 \big \}. \]
On that event, we have $\|\hat u_N\|_{H^{\alpha+1}} \le \sqrt K$, and choosing $L_0 \ge M\ge \|u_0\|_{H^{\alpha+1}} $ large enough without loss of generality, we have 
    \begin{equation*}
        \begin{split}
            \|\hat u_N-u_0\|_{L^2}^2 &\ge \|\hat u_N-u_0\|_{L^2}^2- \|\hat u_N-u_0\|_{N}^2 \\
            &= \mathcal T^2((\hat u_N,\hat f_N),u_0) -\tau^2((\hat u_N,\hat f_N),u_0)\\
            &\ge (K^2-K)\mu_N^2 \ge K^2\mu_N^2/4 \ge K \|\hat u_N-u_0\|_{N}^2/4,
        \end{split}
    \end{equation*}
    and at the same time
    \[ \|\hat u_N-u_0\|_{H^{\alpha+1}} \le L_0 + \sqrt K \le 2\sqrt K.  \]
Thus, it holds that 
\[A_N \subseteq \Big\{
\exists v\in H^{\alpha+1}(\mathcal O): \|v\|_{H^{\alpha+1}}\le 2\sqrt K,\ 
\|v\|_{L^2}\ge K\mu_N/2,\ \|v\|_{L^2}\ge \sqrt K \|v\|_N/2
\Big\}.\]
Since $\alpha+1>d/2$, Sobolev embedding yields that for some $M_K>0$, $\sup_{w:\|w\|_{H^{\alpha+1}}\le 2\sqrt K}\|w\|_\infty\le M_K\lesssim \sqrt K$. Using Lemma \ref{lem-emp-norm} with $R= K\mu_N/2$ and $M=M_K \asymp\sqrt K$ (noting that $\sqrt K/2 \ge 2$), it follows that for some constant $C>0$,
    \[ \P_0^N(A_N) \le 2\exp(-N\mu_N^2K/C),~~~ K>0~~\text{large enough}.\]
    Combining this with the concentration result for the empirical norm \eqref{empirical-rate}, we conclude that for some $C>0$ and all $K\ge K_0$ large enough,
    \begin{align*}
        \P_0^N\big( \mathcal T^2((\hat u_N,\hat f_N),u_0) \ge K^2 \mu_N^2\big) &\le \P_0^N\big( A_N \big) +  \P_0^N\big( \tau^2((\hat u_N,\hat f_N),u_0) \ge K\mu_N^2 \big) \\
        &\le C \exp(-N\mu_N^2K/C).
    \end{align*}
To obtain the corresponding moment bounds, let us write for simplicity 
$Z=\mathcal T^{2}((\hat u_N,\hat f_N),u_0)$. From the previous displayed inequality, it follows that we have
\[
\P_0^N(Z^k \ge t)=\P_0^N(Z\ge t^{1/k}) 
\lesssim \exp\big(-N\mu_N t^{1/2k} /C\big),
\qquad \forall k\ge 1,\ \forall t\ge (K_0\mu_N)^{2k}.
\]
Using Lemma~E.2 with $\gamma=1/(2k)$, $L=N\mu_N/C$ and $a=(K_0\mu_N)^{2k}$ 
(note that $La^\gamma=(N\mu_N/C)\cdot (K_0\mu_N)\simeq K_0N\mu_N^2\ge 1$ since $K_0\ge 1$), 
we deduce that
\begin{equation*}
    \begin{split}
    \E_0^N \big[Z^k\big]
    &\lesssim  (K_0\mu_N)^{2k} 
    +\int_{(K_0\mu_N)^{2k}}^\infty 
    \exp\big(-N\mu_N t^{1/2k} /C\big)\, dt \\
    &\lesssim \mu_N^{2k} 
    +\frac{\mu_N^{2k(1-\frac{1}{2k})}}{N\mu_N} 
    \exp\big(-K_0N\mu_N^2/C\big)\\
    &\le \mu_N^{2k} + \frac{\mu_N^{2k-2}}{N} 
    = O(\mu_N^{2k}),
    \end{split}
\end{equation*}
where we used in the last step that $\mu_N = o(\sqrt N)$.

\paragraph*{Step 3: Convergence rate for the inverse problem} Using the definitions of $\mathcal R$ and of $\mathcal T^2$, the preceding bound implies that for any $k\ge 1$, there exists $C_k>0$ such that
    \begin{equation}\label{darcy-moments}
    \begin{split}
                \E_0^N [ \|\hat f_N\|_{H^{\alpha}}^{2k}] + \E_0^N [ \|\hat u_N\|_{H^{\alpha+1}}^{2k}]&\le C_k,~~~ \forall N\ge 1,
               % \E_0^N [ \|L_{\hat f_N}\hat u_N-g\|_{L^2}^{2k}] &\le C_k N^{-\frac{2k(\alpha-1)}{2(\alpha+1)+d}},
    \end{split}
    \end{equation}
    as well as 
    \begin{equation}\label{darcy-MSE-moments}
        \E_0^N [ \|\hat u_N-u_0\|_{L^2}^{2k}]\le C_k \mu_N^{2k},~~~ \forall N\ge 1.
   \end{equation}
    A standard Sobolev interpolation argument gives 
    \begin{equation}\label{interpol}
        \begin{split}
        \|\hat u_N -u_0\|_{H^2}&\lesssim \|\hat u_N -u_0\|_{L^2}^{\frac{\alpha-1}{\alpha+1}} \|\hat u_N -u_0\|_{H^{\alpha+1}}^{\frac{2}{\alpha+1}} \\
    &\lesssim 
    \|\hat u_N -u_0\|_{L^2}^{\frac{\alpha-1}{\alpha+1}} \Big( \|\hat u_N \|_{H^{\alpha+1}}^{\frac{2}{\alpha+1}} + \|u_0\|_{H^{\alpha+1}}^{\frac{2}{\alpha+1}}\Big).
        \end{split}
    \end{equation}
We next verify that we may utilize Lemma~\ref{darcy-generalized-stability} with $(u_1,f_1,g_1)=(u_0,f_0,g)$ and $(u_2,f_2,g_2)=(\hat u_N,\hat f_N,L_{\hat f_N}\hat u_N)$. Since $f_0\in \mathcal F(\alpha,f_{\min},R)$, we have $f_0\ge f_{\min}>0$ on $\mathcal O$ and, by Sobolev embedding
$H^\alpha(\mathcal O)\hookrightarrow C^1(\overline{\mathcal O})$ (recall $\alpha>d/2+1$), also
$\|f_0\|_{C^1(\mathcal O)}\lesssim \|f_0\|_{H^\alpha(\mathcal O)}\le R$.  
Moreover, by Lemma~\ref{lem-darcy-boundedness} the corresponding solution $u_0=\mathcal G(f_0)$ satisfies
$u_0\in H^{\alpha+1}(\mathcal O)\subseteq C^2(\overline{\mathcal O})$ and $\|u_0\|_{C^2(\mathcal O)}\le B$ for some
$B=B(\alpha,f_{\min},R,\mathcal O)$.  
Finally, the boundary condition $u_0|_{\partial\mathcal O}=0$ holds by definition of $\mathcal G$. Consequently, Lemma~\ref{darcy-generalized-stability} yields
    \[ \|\hat f_N-f_0\|_{L^2}^2 \lesssim \|\hat f_N\|_{C^1}^2 \|\hat u_N-u_0\|_{H^2}^2 + \|L_{\hat f_N}\hat u_N - g\|_{L^2}^2 =: I+II \]
    We separately bound the expectation of the two terms. For the second term, we obtain by the definition of $\mathcal R$ and the choice of $\lambda_N$ from \eqref{Darcy-hyperparams} that
    \[ \E_0^N [II] \lesssim \lambda_N^{-2} \E_0^N[ \mathcal R(\hat u_N,\hat f_N) ]  \lesssim \lambda_N^{-2} \E_0^N[ \mathcal T^2((\hat u_N,\hat f_N),u_0) ] \lesssim N^{-\frac{2(\alpha-1)}{2(\alpha+1)+d}}. \]

Since $\alpha>d/2+1$, Sobolev embedding yields $H^\alpha(\mathcal O)\hookrightarrow C^1(\overline{\mathcal O})$. By \eqref{darcy-moments}, it thus holds that $\sup_{N\ge1}\E_0^N\|\hat f_N\|_{C^1(\mathcal O)}^{2k}<\infty$ for every $k\ge1$. Using this, the Cauchy-Schwarz inequality, \eqref{interpol} and \eqref{darcy-moments},
    \begin{align}\notag
        \E_0^N [I] &\le  \big( \E_0^N \|\hat f_N\|_{C^1}^4 \big)^{1/2} \big(\E_0^N \|\hat u_N-u_0\|_{H^2}^4\big)^{1/2}\\ 
           \notag &\lesssim \big( \E_0^N \|\hat f_N\|_{C^1}^4 \big)^{1/2} \Big(\E_0^N   \|\hat u_N -u_0\|_{L^2}^{\frac{4(\alpha-1)}{\alpha+1}} \Big( \|\hat u_N \|_{H^{\alpha+1}}^{\frac{8}{\alpha+1}} + \|u_0\|_{H^{\alpha+1}}^{\frac{8}{\alpha+1}}\Big) \Big]\Big)^{1/2}\\ 
           \label{TermI} &\lesssim \Big( \E_0^N  \Big[ \|\hat u_N -u_0\|_{L^2}^{\frac{4(\alpha-1)}{\alpha+1}} \|\hat u_N \|_{H^{\alpha+1}}^{\frac{8}{\alpha+1}}\Big]\Big)^{1/2}  + \|u_0\|_{H^{\alpha+1}}^{\frac{4}{\alpha+1}} \Big(\E_0^N   \|\hat u_N -u_0\|_{L^2}^{\frac{4(\alpha-1)}{\alpha+1}}\Big)^{1/2}\\
           &=: I_a + I_b. \notag
    \end{align}
Term $I_b$ is estimated using Jensen's inequality and (\ref{darcy-MSE-moments}),
\[ I_b \lesssim \E_0^N  \Big[ \|\hat u_N -u_0\|_{L^2}^4\Big]^{\frac{\alpha-1}{2(\alpha+1)}}\lesssim \mu_N^{\frac{2(\alpha-1)}{\alpha+1}} \lesssim N^{-\frac{2(\alpha-1)}{2(\alpha+1)+d}}. \]
For term $I_a$, we use H\"older's inequality $\|vw\|_{L^2} \lesssim \|v\|_{L^p} \|w\|_{L^q}$ with 
\[p= \frac{2(\alpha+1)}{\alpha-1},~~q=\alpha+1,~~\text{with}~~\frac 1p+\frac 1q =1/2,\] as well as \eqref{darcy-moments}-\eqref{darcy-MSE-moments} to obtain 
\[ I_a \lesssim \big( \E_0^N  \|\hat u_N -u_0\|_{L^2}^4\big)^{\frac{\alpha-1}{2(\alpha+1)}} \big( \E_0^N   \|\hat u_N \|_{H^{\alpha+1}}^{4} \big)^{\frac{1}{\alpha+1}}\lesssim \mu_N^{\frac{2(\alpha-1)}{\alpha+1}} \lesssim N^{-\frac{2(\alpha-1)}{2(\alpha+1)+d}}.\]
Putting everything together, we obtain that
\[ E_0^N \big[ \|\hat f_N-f_0\|_{L^2}^2 \big] = O(N^{-\frac{2\alpha-2}{2\alpha+2+d}}),  \]
which finishes the proof of the convergence rate for the PDE-penalized M-estimator.

\subsubsection{Plug-in estimator}\label{sec-plugin-infdim-proof} We now prove the analogous convergence guarantee for the plug-in estimator.

\paragraph*{Step 1: Convergence of $\hat u_N$} We again employ Theorem \ref{thm-random-design}, this time with a single parameter $\eta= u\in H^{\alpha+1}$. [Formally we may take $\Theta=\{0\}$ such that the minimization problem is independent of $\theta$.] The regularization functional is given by
\[\mathcal R(u)= {\mu}_N^2\|u\|_{H^{\alpha+1}}^2. \]
Define the functional 
\[ \tau^2(\hat u_N,u_0) := \|\hat u_N -u_0\|_N^2+  \mu_N^2\|u\|_{H^{\alpha+1}}^2, \]
and set $u^* = u_0$. By the definition of $\mathcal R$, it holds that
\[\mathcal U(\eta^*,R) \subseteq \{ u: \|u\|_{H^{\alpha+1}} \le R/\mu_N \},~~R>0.\]
Thus, the same arguments \eqref{entropy-integral-bound}-\eqref{Psi-definition} as in the proof for the generalized M-estimator give that the hypotheses of Theorem \ref{thm-random-design} are fulfilled with $\Psi (R) := C_1R\mu_N^{-\frac{d}{2(\alpha+1)}}$, and with $\delta_N := C_2 N^{-\frac{\alpha+1}{2(\alpha+1)+d}}$, for $C_1,C_2>0$ large enough. This implies that we have the concentration inequality
    \begin{equation}\label{empirical-rate-plugin}
        \P_0^N \Big(\tau^2(\hat u_N,u_0) \ge L^2N^{-\frac{2(\alpha+1)}{2(\alpha+1)+d}} \Big)\le C_3\exp\big(-L^2N^{\frac{d}{2(\alpha+1)+d}}/C_3\big),~~~L\ge L_0,
    \end{equation}
    where $L_0>0$ is some large enough constant, as well as the moment bound
    \[ \E_0^N \big[\tau^{2k}(\hat u_N,u_0)\big] \lesssim  \delta_N^{2k},~~~\text{any}~k\ge 1. \]
    The preceding bound extends also to the $L^2(\mathcal O)$-norm. Arguing just as in the proof for the PDE-penalized M-estimator, defining \[\mathcal T^2(\hat u_N,u_0)=:\|\hat u_N-u_0\|_{L^2}^2+ \mu_N^2 \|u\|_{H^{\alpha+1}}^2,\] and application of Lemma \ref{lem-emp-norm} yields that for all $K\ge K_0$ large enough,
    \begin{align*}
        \P_0^N\big( \mathcal T^2(\hat u_N,u_0) \ge K^2 N^{-\frac{2(\alpha+1)}{2(\alpha+1)+d}} \big)
        &\lesssim  \exp(-KN^{\frac{d}{2(\alpha+1)+d}}/C),\\
        \E_0^N\big[ \mathcal T^{2k}(\hat u_N,u_0) \big]&\lesssim \mu_N^{2k} =N^{-\frac{2k(\alpha+1)}{2(\alpha+1) +d}},
    \end{align*}
which implies in particular that for some $C_k>0$,
\begin{equation}\label{Darcy-plug-in-bounds}
    \E_0^N\big[ \|\hat u_N-u_0\|_{L^2}^{2k} \big] \le C_k\mu_N^{2k},~~~\text{and}~~~\E_0^N\big[ \|\hat u_N\|_{H^{\alpha+1}}^{2k} \big] \le C_k,~~~\forall N\ge 1.
\end{equation}

\paragraph*{Step 2: Convergence of $\hat f_N$} By definition of $\hat f_N$, it holds that 
\begin{align*}
    \|\nabla \cdot (\hat f_N \nabla \hat u_N)-g \|_{L^2}^2 + \nu_N^2 \|\hat f_N\|_{H^\alpha}^2 &\le \|\nabla \cdot (f_0 \nabla \hat u_N)-g \|_{L^2}^2 + \nu_N^2 \|f_0\|_{H^\alpha}^2\\
    &\le \|\nabla \cdot (f_0 \nabla (\hat u_N- u_0))\|_{L^2}^2 + \nu_N^2 \|f_0\|_{H^\alpha}^2\\
    &\lesssim \|f_0\|_{C^1}^2 \|\hat u_N- u_0\|_{H^2}^2 +\nu_N^2 \|f_0\|_{H^\alpha}^2\\
    &\lesssim \|\hat u_N- u_0\|_{H^2}^2+\nu_N^2.
\end{align*}
Note that the right hand side is always finite, since $\hat u_N\in H^{\alpha+1}$ and $\alpha\ge 1$ by assumption. The preceding inequality in particular implies that 
\[\|\hat f_N\|_{H^\alpha}^2 \lesssim 1+ \nu_N^{-2}\|\hat u_N- u_0\|_{H^2}^2.\]
As in the proof of the PDE-penalized estimator, the hypotheses of the generalized stability estimate from Lemma~\ref{darcy-generalized-stability}
are satisfied with $(u_1,f_1,g_1)=(u_0,f_0,g)$, since $f_0\ge f_{\min}$ and $u_0\in C^2(\overline{\mathcal O})$
with $\|f_0\|_{C^1}\vee \|u_0\|_{C^2}\le B$. Using Lemma \ref{darcy-generalized-stability}, the Sobolev embedding $H^\alpha(\mathcal O)\hookrightarrow C^1(\overline{\mathcal O})$ (since $\alpha>d/2+1$), as well as Cauchy-Schwarz, we obtain
\begin{align}\notag
    \E_0^N \big[\|\hat f_N -f_0\|_{L^2}^2 \big]&\lesssim \E_0^N \big[ \|\hat f_N\|_{C^1}^2 \|\hat u_N -u_0\|_{H^2}^2 + \|\nabla \cdot(\hat f_N\nabla \hat u_N)-g \|_{L^2}^2 \big]\\ \notag
    &\lesssim \E_0^N \big[ (1+\|\hat f_N\|_{C^1}^2) \|\hat u_N -u_0\|_{H^2}^2 \big]\\ \notag
    &\lesssim \E_0^N \big[ (1+\|\hat f_N\|_{H^\alpha}^2) \|\hat u_N -u_0\|_{H^2}^2 \big]\\ \label{plugin-inverse-rate}
    &\lesssim \Big(\E_0^N \big[1+ \eta_N^{-4}\|\hat u_N -u_0\|_{H^2}^4 \big]\Big)^{1/2}  \Big(\E_0^N \big[\|\hat u_N -u_0\|_{H^2}^4 \big]\Big)^{1/2}.
\end{align}
Using the same interpolation argument following \eqref{TermI} as well as the moment bounds in \eqref{Darcy-plug-in-bounds}, we can now estimate
\begin{align*}
    &\Big(\E_0^N \big[\|\hat u_N -u_0\|_{H^2}^4  \big]\Big)^{1/2}\\
    &\lesssim \Big( \E_0^N  \Big[ \|\hat u_N -u_0\|_{L^2}^{\frac{4(\alpha-1)}{\alpha+1}} \|\hat u_N \|_{H^{\alpha+1}}^{\frac{8}{\alpha+1}}\Big]\Big)^{1/2}  + \|u_0\|_{H^{\alpha+1}}^{\frac{4}{\alpha+1}} \Big(\E_0^N   \|\hat u_N -u_0\|_{L^2}^{\frac{4(\alpha-1)}{\alpha+1}}\Big)^{1/2}\\
    &\lesssim N^{-\frac{2(\alpha-1)}{2(\alpha+1)+d}} = \nu_N^2.
\end{align*}
Plugging the above into both factors on the right hand side of \eqref{plugin-inverse-rate} yields that 
\begin{equation*}
       \E_0^N \big[\|\hat f_N -f_0\|_{L^2}^2 \big]\lesssim N^{-\frac{2(\alpha-1)}{2(\alpha+1)+d}},
\end{equation*}
which completes the proof of theorem. 
\hfill \qed

\subsection{Proof of Theorem \ref{thm-darcy-polytime}}

\paragraph*{Step 1: Convergence of $\hat u_N$}
Let $\eta_0=S(u_0)\in h^{\alpha+1}$ denote the frame coefficients of $u_0$ satisfying 
\[  \|u_0\|_{H^{\alpha+1}(\mathcal O)} \simeq\|\eta_0\|_{h^{\alpha+1}}, \]
and denote by $\eta_{0,J}\in\R^{p_J}$ its truncation at resolution level $J$. For
\[u_{0,J}:=\sum_{l\le J,k\le N_l}(\eta_{0,J})_{lk}\phi_{lk}\in L^2_J(\mathcal O),\]
standard approximation properties of the $C^u$-regular wavelet frame yield
\begin{equation}\label{eq:approx-u}
\|u_0-u_{0,J}\|_{L^2}\lesssim 2^{-J(\alpha+1)}\|u_0\|_{H^{\alpha+1}}.
\end{equation}
We will apply Theorem~\ref{thm-random-design} for the high-dimensional parameter $\eta\in H=\R^{p_J}$, with parameterization
\[u_\eta= \sum_{l\le J,k\le N_l} \eta_{lk}\phi_{lk}, \]
with $\eta^*= \eta_{0,J}$ and with penalty term $\mathcal R(\eta)=\mu_N^2\|\eta\|_{h^{\alpha+1}_J}^2$. Recall the definition 
\[\tau^2(\eta , u_0)= \|u_0-u_\eta\|_N^2+\mu_N^2\|\eta\|_{h^{\alpha+1}_J}^2. \]
Similarly to the infinite-dimensional case, the localized function classes are contained in bounded Sobolev balls. For some constant $c>0$, we have
\[\mathcal U(\eta^*,R) \subseteq \{ u: \|u\|_{H^{\alpha+1}} \le cR/\mu_N \},~~R>0.\]
Thus, using again the entropy bounds from Theorem 4.10.3 of \cite{Triebel1978}, the hypotheses of Theorem \ref{thm-random-design} are fulfilled with $\Psi (R) := C_1R\mu_N^{-\frac{d}{2(\alpha+1)}}$, and with $\delta_N := C_2 N^{-\frac{\alpha+1}{2(\alpha+1)+d}}$, for some $C_1,C_2>0$. It follows from Theorem \ref{thm-random-design} that for all $R\ge \delta_N$ and some $C>0$,
\begin{equation}\label{eq:u-oracle}
\P_{f_0}^N\Big(
\tau^2(\hat \eta_N, u_0)
\ge
2\big(\|u_0-u_{0,J}\|_N^2+\mu_N^2\|\eta_{0,J}\|_{h^{\alpha+1}_J}^2 +R^2\big)
\Big)
\le C
\exp(-NR^2/C).
\end{equation}
An application of Bernstein's inequality to the centred random variables (recalling that $\text{vol}(\mathcal O)=1$)
\begin{equation*}
    \begin{split}
        Z_i&=(u_0-u_{0,J})^2(X_i)-\|u_0-u_{0,J}\|_{L^2}^2,\\ |Z_i|&\le \|u_0-u_{0,J}\|_\infty^2 \lesssim 2^{-2J(\alpha+1-d/2)}\|u_0\|_{H^{\alpha+1}}^2 \lesssim \mu_N^2 2^{Jd},\\
        \E [Z_i^2] &\lesssim \|u_0-u_{0,J}\|_{L^2}^2\|u_0-u_{0,J}\|_{\infty}^2\lesssim \mu_N^42^{Jd},
    \end{split}
\end{equation*}
gives that for any $L\ge 1$, and some $C,C',C''>0$,
\begin{equation*}
    \begin{split}
    \P_{0}^N\big(\|u_0-u_{0,J}\|_N^2-\|u_0-u_{0,J}\|_{L^2}^2\ge L^2\mu_N^2 \big)&\le 2\exp\Big(-\frac{NL^4\mu_N^4}{C(1+L^2) 2^{Jd}\mu_N^4 }\Big)\\
    &\lesssim \exp(-N L^2/2^{Jd}C')\\
    &\lesssim \exp(-\mu_N^{-2}L^2/C'')\\
    &\lesssim \exp(-N\mu_N^2L^2/C'').
    \end{split}
\end{equation*}
Here, in the final step, we also used that $\mu_N=N^{-\frac{\alpha+1}{2(\alpha+1)+d}}\le N^{-1/4}$ due to our assumption $\alpha > d/2+1$.

Combining this with \eqref{eq:u-oracle}
and noting that $\|\eta_{0,J}\|_{h^{\alpha+1}_J}\lesssim \|u_0\|_{H^{\alpha+1}}$ which is bounded by a constant, it follows that for all $K\ge K_0$ and some $C>0$, 
\begin{equation*}
\P_{f_0}^N\big(
\tau^2(\hat \eta_N, u_0)
\ge
K^2\mu_N^2
\big)
\le C
\exp(-NK^2\mu_N^2/C).
\end{equation*}
By arguing as in the proof of Theorem \ref{Darcy-rate}, this implies in particular the moment bounds
\begin{equation}\label{eq:u-moment}
    \E_0^N\big[ \|\hat u_N-u_0\|_{L^2}^{2k} \big] \le C_k\mu_N^{2k},~~~\text{and}~~~\E_0^N\big[ \|\hat u_N\|_{H^{\alpha+1}}^{2k} \big] \le C_k,~~~\forall N\ge 1.
\end{equation}
Now define the $L^2$-counterpart of the $\tau^2$-functional,
\[\mathcal T^2(\eta , u_0)= \|u_0-u_\eta\|_{L^2}^2+\mu_N^2\|\eta\|_{h^{\alpha+1}}^2. \]
The same application of Lemma \ref{lem-emp-norm} as in Step 2 of the proof of Theorem \ref{Darcy-rate} gives that for all $K\ge K_0>0$ large enough, and some $C>0$,
\begin{align*}
        \P_0^N\big( \mathcal T^2((\hat u_N,\hat f_N),u_0) \ge K^2 \mu_N^2\big) \le C  \exp(-N\mu_N^2K/C),
\end{align*}
as well as, for some $C_k>0$,
\begin{equation}\label{high-dim-plugin-moments}
    \E_0^N\big[ \|\hat u_N-u_0\|_{L^2}^{2k} \big] \le C_k\mu_N^{2k},~~~\text{and}~~~\E_0^N\big[ \|\hat u_N\|_{H^{\alpha+1}}^{2k} \big] \le C_k,~~~\forall N\ge 1.
\end{equation}

\paragraph*{Step 2: Convergence of $\hat f_N$}

Let $\theta_0 = S(f_0)\in h^{\alpha}$ and $\theta_{0,J}\in \R^{p_J}$ denote the frame coefficients (and their truncated counterpart) of $f_0$, and let 
\[f_{0,J}:=\sum_{l\le J,k\le N_l}(\theta_{0,J})_{lk}\phi_{lk}\in L^2_J(\mathcal O).\]
Then we have that 
\[ \|f_0\|_{H^{\alpha}(\mathcal O)} \simeq\|\theta_0\|_{h^{\alpha}},~~ \|f_0- f_{0,J}\|_{L^2}^2\lesssim 2^{-2J\alpha} \|f_0\|_{H^{\alpha}}^2,~~ \|f_0- f_{0,J}\|_{H^1}^2\le 2^{-2J(\alpha-1)} \|f_0\|_{H^{\alpha}}^2. \]
By definition of $\hat \theta_N$, we have the `basic inequality'
\begin{equation}\label{plugin-basic-ineq}
    \begin{split}
        \|\Psi\hat \theta_N -\gamma \|_{\R^{p_J}}^2 + \nu_N^2 \|\hat \theta_N\|_{h^\alpha}^2  \le  \|\Psi \theta_{0,J} -\gamma \|_{\R^{p_J}}^2 + \nu_N^2  \| \theta_{0,J} \|_{h^\alpha}^2.
    \end{split}
\end{equation}
The second term on the right hand side is upper bounded by
\[ \nu_N^2  \| \theta_{0,J} \|_{h^\alpha}^2  \le \nu_N^2  \| \theta_{0} \|_{h^\alpha}^2 \simeq \nu_N^2 \|f_0\|_{H^\alpha}^2 \lesssim N^{-\frac{2(\alpha-1)}{2(\alpha+1)+d}}. \]
For the first term, let us denote the coefficient selection operator up to level $J$ by
\[ S^{(J)} : H^\alpha (\mathcal O) \to \R^{p_J},\]
and let us also introduce the associated `pseudo-projection operator' 
\[ P^{(J)}: L^2(\mathcal O) \to L^2_J(\mathcal O),~~~ P^{(J)} (f)= \sum_{l\le J,k} S(f)_{lk}\phi_{lk}. \]
We then estimate
\begin{equation*}
\begin{split}
 \|\Psi \theta_{0,J} -\gamma \|_{\R^{p_J}} &= \| S^{(J)}(\nabla\cdot(f_{0,J}\nabla\hat u_N)- g ) \|_{\R^{p_J}} \\
 &\le \| S(\nabla\cdot(f_{0,J}\nabla\hat u_N)- g ) \|_{\ell^2(\N)} \\
 &\simeq  \| \nabla\cdot(f_{0,J}\nabla\hat u_N)- g  \|_{L^2(\mathcal O)} \\
\end{split}
\end{equation*}
Using $L_{f_0}u_0=g$, the $C^2(\mathcal O)$-bound from Lemma \ref{lem-darcy-boundedness}, and recalling that $2^{-J(\alpha-1)}\simeq \nu_N = N^{-\frac{\alpha-1}{2(\alpha+1)+d}}$, we further estimate this by
\begin{equation*}
\begin{split}
	\| \nabla\cdot(f_{0,J}\nabla\hat u_N)& - \nabla\cdot(f_0\nabla u_0) \|_{L^2(\mathcal O)}\\
	&\le \| \nabla\cdot(f_{0,J}\nabla(\hat u_N - u_0)) \|_{L^2} +\| \nabla\cdot((f_{0,J}- f_{0}) \nabla u_0) \|_{L^2}\\
	&\lesssim  \|f_{0,J} \|_{C^1} \|\hat u_N-u_0\|_{H^2} + \| f_{0,J}- f_{0}\|_{H^1} \| u_0 \|_{C^2}\\
	&\lesssim \|\hat u_N-u_0\|_{H^2} + 2^{-J(\alpha-1)} \|f_0\|_{H^\alpha}\\
	&\lesssim  \|\hat u_N-u_0\|_{H^2}  + \nu_N.
	\end{split}
\end{equation*}
Using the previous steps, the fact that $H^{\alpha-1}$ is a multiplication algebra satisfying the inequality
\[ \| vw \|_{H^{\alpha-1}}\lesssim \| v \|_{H^{\alpha-1}}\| w \|_{H^{\alpha-1}} ~~~\forall v,w\in H^{\alpha-1}. \]
as well as the Sobolev interpolation inequality, we now conclude that
\begin{equation*}
	\begin{split}
		\| L_{\hat f_N} \hat u_N - g \|_{L^2} 
        & %\| L_{\hat f_N} \hat u_N - L_{f_0}u_0 \|_{L^2} \\
		\lesssim \| S^{(J)}(L_{\hat f_N} \hat u_N - g) \|_{\R^{p_J}} +2^{-J(\alpha-1)} \| L_{\hat f} \hat u_N - g\|_{H^{\alpha-1}}\\
		&=  \|\Psi \hat\theta_N - \gamma\|_{\R^{p_J}} + 2^{-J(\alpha-1)} \| L_{\hat f} \hat u_N - g\|_{H^{\alpha-1}}\\
		&\lesssim  \|\Psi \hat\theta_N - \gamma\|_{\R^{p_J}} + \nu_N\big( \|\hat f_N\|_{H^\alpha}\|\hat u_N \|_{H^{\alpha+1}} + \|g\|_{H^{\alpha-1}}  \big)  \\
		&\lesssim \|\hat u_N-u_0\|_{H^2}  + \nu_N \big(1+\|\hat f_N\|_{H^\alpha}\|\hat u_N \|_{H^{\alpha+1}} \big)\\
        & \lesssim \big(1+\|\hat u_N \|_{H^{\alpha+1}} +\|\hat f_N\|_{H^\alpha}\|\hat u_N \|_{H^{\alpha+1}} \big)\Big(\|\hat u_N-u_0\|_{L^2}^{\frac{\alpha-1}{\alpha+1}}  +  \nu_N\Big).
		\end{split}
\end{equation*}
To further estimate $\|\hat f_N\|_{H^\alpha}$, note that the above arguments and the `basic inequality' \eqref{plugin-basic-ineq} also imply
\begin{equation*}
    \nu_N^2 \|\hat f_N\|_{H^\alpha}^2 \le \|\nabla \cdot (f_{0,J}\nabla \hat u_N)-g \|_{L^2}^2 + \nu_N^2 \|f_{0,J}\|_{H^\alpha}^2
    \lesssim \|\hat u_N- u_0\|_{H^2}^2+\nu_N^2.
\end{equation*}
An application of the generalized stability Lemma \ref{darcy-generalized-stability} with the triples $(f_0, u_0, g)$ and $(\hat f,\hat u_N, L_{\hat f_N} \hat u_N)$, which is justified since $f_0,u_0,g$ satisfy the required `ellipticity' conditions, and putting everything together, we obtain that
\begin{equation*}
	\begin{split}
	\|\hat f_N- f_0\|_{L^2}&\lesssim \|\hat f_N\|_{C^1} \| \hat u_N-u_0\|_{H^2} + \| L_{\hat f_N} \hat u_N - g\|_{L^2}\\
	&\lesssim \|\hat f_N \|_{H^\alpha}\| \hat u_N-u_0\|_{H^2} \\
    &\qquad+ \big(1+\|\hat u_N \|_{H^{\alpha+1}} \big) \big(1+\nu_N^{-1}\| \hat u_N-u_0\|_{H^2}
    \big)\Big(\|\hat u_N-u_0\|_{L^2}^{\frac{\alpha-1}{\alpha+1}}  +  \nu_N\Big)\\
    &=:I+II.
	\end{split}
\end{equation*}
Then, using the moment bounds \eqref{high-dim-plugin-moments}, we may perform an argument on each of the terms $I$ and $II$, akin to the argument in Step 2 of the proof for the infinite-dimensional case in Section \ref{sec-plugin-infdim-proof} following the display \eqref{Darcy-plug-in-bounds}, to conclude the proof and obtain that 
\[ \E_0^N\big[\|\hat f_N-f_0\|_{L^2}^2 \big] \lesssim \nu_N^2, \]
which completes the proof of Theorem \ref{thm-darcy-polytime}. \hfill\qed

\subsection{Proof of Theorem \ref{darcy-initialisation}}\label{sec-proof-init}
	 Let $\hat f_N$ denote the plug-in estimator from Theorem \ref{thm-darcy-polytime}, computable in polynomial time. It follows from the proof of Theorem \ref{thm-darcy-polytime} that for $L\ge 1$ sufficiently large, we have
	\[\P_0^N \big( \|\hat f_N\|_{H^\alpha}\le L \big) \xrightarrow{N\to\infty} 1. \]
	Moreover, since $\alpha >d/2 +1$, we have that on the event $\{  \|\hat f_N\|_{H^\alpha}\le L  \}$, the Sobolev embedding $H^{\alpha-1}\hookrightarrow L^\infty$ and interpolation imply that
	\[ \|\hat f_N-f_0\|_\infty \lesssim \|\hat f_N-f_0\|_{H^{\alpha-1}} \lesssim \|\hat f_N-f_0\|_{L^2}^{\frac{1}{\alpha}}\|\hat f_N-f_0\|_{H^{\alpha}}^{\frac{\alpha-1}{\alpha}}\le L\|\hat f_N-f_0\|_{L^2}^{\frac{1}{\alpha}}.
	\]
	By Theorem \ref{thm-darcy-polytime}, we conclude that
	\[  \|\hat f_N-f_0\|_\infty \xrightarrow{N\to\infty} 0 ~~\text{in}~~\P_0^N\text{-probability}.\]
	In particular, setting $\rho := \inf_{x}f_0(x) -f_{min}>0$ and choosing $L$ large enough, it holds that the event 
	\[ A_N:= \{   \inf_{x} \hat f_N(x) \ge f_{min}+\rho/2 \} \cap \{ \|\hat f_N\|_{H^\alpha}\le L  \}\cap \{ \|\hat f_N-f_0\|_{L^2}\le L N^{-\frac{\alpha-1}{2(\alpha+1)+d}} \}. \]
	has probability tending to $1$, $\P_0^N(A_N)\to 1$, as $N\to\infty$. Under our assumptions on the link function $\Phi$, its inverse  $\Phi^{-1}$ restricted to the interval $[f_{min}+\rho/2,\infty)$ is uniformly Lipschitz with uniformly bounded derivatives of all orders, and thus we have the Sobolev estimate 
	\[ \sup_{\|f\|_{H^\alpha}\le L,~f\ge f_{min}+\rho/2} \|\Phi^{-1}\circ f\|_{H^\alpha } \le M<\infty \] 
for some finite constant $M=M(\alpha,\Phi,f_{min},\rho,L,\mathcal O)$. See, for instance, \cite[Lemma 30]{NVW18} for a proof, which uses Faa di Bruno's formula as well as the Gagliardo-Nirenberg interpolation inequality. We set $\hat F_N:= \Phi^{-1}\circ \hat f_N$ and $F_0:= \Phi^{-1}\circ f_0$ (outside of the event $A_N$, we may set $\hat F_N$ arbitrarily, in a measurable way). Then on the event $A_N$, it holds that $\|\hat F_N\|_{H^\alpha} \le M$, and moreover that 
\[ \|\hat F_N - F_0\|_{L^2}\lesssim \|\hat f_N - f_0\|_{L^2} \le LN^{-\frac{\alpha-1}{2(\alpha+1)+d}}. \] 
Finally we define $\theta_{init}$ to be given by the first $D$ coefficients in the $(e_k:k\ge 1)$-basis of $L^2(\mathcal O)$, 
\[ \theta_{init, k} = \langle \hat F_N, e_k\rangle_{L^2}~~~\forall 1\le k\le D,~~~ \theta_{init}\in \R^D. \]
By Parseval's identity for the orthonormal basis $(e_k:k\ge 1)$, we then have
\begin{equation*}
	\begin{split}
		\|\theta_{init}-\theta_{0,D} \|_{\R^D} & =\Big(\sum_{k=1}^D \big|\langle \hat F_N-F_0,e_k\rangle_{L^2}\big|^2\Big)^{1/2}\\
        &\le \|\hat F_N-F_0\|_{L^2}\lesssim LN^{-\frac{\alpha-1}{2(\alpha+1)+d}}.
	\end{split}
\end{equation*}
For any $\alpha \ge 12$ (and in particular for any $\alpha\ge 22$, as assumed),  the latter rate is of the order $o(  D^{-8/d}/\log D)$, since $D^{-8/d}= N^{-\frac{8}{2\alpha+d}}$ and
\[ \frac{\alpha-1}{2(\alpha+1)+d} \ge \frac{\alpha-3}{2\alpha+d} > \frac{8}{2\alpha+d}, \]
which completes the proof.

\hfill \qed

\subsection{Proof of Theorem \ref{thm-adaptive}}\label{sec-adaptive-proof}
    \paragraph*{Step 1} We first prove \eqref{adaptive-u-consistency} for the empirical norm, using Theorem \ref{thm-random-design} with $\eta = (\beta,u)\in H$ and parameterization $\eta=(\beta,u)\mapsto u$, where
    \[ H = \big\{ (\beta,u) : \beta\in [\beta_{\min}, \beta_{\max}],~u\in H^\beta (\mathcal O)\big\}. \]
    %We notice that our estimated pair $(\hat \beta,\hat u)$ yields the same regression function $\hat u$ as the least squares problem with penalty
    %\[ \mathcal R(u) = \inf_{\beta}\mu_{N,\beta}^2\|u\|_{H^\beta}^2. \]
    Set $\eta^*= (\beta_0,u_0)$, and consider the for $R>0$ the localized function classes
    \[ \mathcal U^*(\eta^*,R)= \{ u: \|\hat u-u_0\|_{N}^2 + \mu_{N,\beta}^2 ( \|u\|_{H^\beta}^2 +A^2 ) \le R^2 ~~\text{for some}~~\beta\in [\beta_{\min},\beta_{\max}] \}. \]
    For any $\beta$ with $\mu_{N,\beta}^2 > R^2/A^2$, there can be no contribution of pairs $(\beta,u)\in H$ with that value of $\beta$. Thus
    \[ \beta(R) :=\inf \{\beta: \mu_{N,\beta} \le R/A \}, \]
    the $\|\cdot\|_\infty$-entropy of $\mathcal U^*(\eta^*,R)$, which upper bounds the empirical entropy, satisfies for some $C>0$
    \begin{equation*}
        H(\mathcal U^*(\eta^*,R), \|\cdot\|_\infty,\rho) \le C  \sum_{\beta\ge \beta(R)} \Big(\frac{R}{\rho\mu_{N,\beta}}\Big)^{\frac{d}{\beta}}\lesssim \Big(\frac{R}{\rho\mu_{N,\beta(R)}}\Big)^{\frac{d}{\beta(R)}}.
    \end{equation*}
    Here, we have used standard entropy bounds for $H^\beta(\mathcal O),~\beta>d/2$ (see Theorem 4.10.3 in \cite{Triebel1978}), the fact that the sum is over a fixed finite number of terms, as well as
    \[ \mu_{N,\beta}^{d/\beta} = N^{\frac{d}{2\beta+d}}\le N^{\frac{d}{2\beta(R)+d}},~~~\text{for}~~~\beta\ge \beta(R).\]
    Noting that $\beta_{min}>d/2$, we can then upper bound the metric entropy integral of $\mathcal U^*(\eta^*,R)$ by 
    \begin{equation*}
        \begin{split}
            \int_0^R H^{1/2}(\mathcal U^*(\eta^*,R), \|\cdot\|_\infty,\rho)d\rho &\lesssim \int_0^R \Big(\frac{R}{\rho\mu_{N,\beta(R)}}\Big)^{\frac{d}{2\beta(R)}} d\rho\\
            &\lesssim R \mu_{N,\beta(R)}^{-\frac{d}{2\beta(R)}} \\
            &= R\sqrt N \mu_{N,\beta(R)} \le R^2 \sqrt N / A.
        \end{split}
    \end{equation*}
    Choosing the constant $A>0$ large enough (only depending $C>0$ as well as the constant from the entropy condition \eqref{entropy-ineq}), it follows that \eqref{entropy-ineq} is fulfilled for all $R>0$. We fix such a choice of $A>0$ henceforth.

    Applying Theorem \ref{thm-random-design}, we obtain that for some $C_1>0$ and all $R>0$
    \begin{equation}\label{adaptive-concentration}
        \begin{split}
        \P_{f_0}^N \Big( \|\hat u_N-u_0\|_{N}^2 + \mu_{N,\hat \beta}^2\|\hat u_N\|_{H^{\hat\beta}}^2  \ge 2 \mu_{N,\beta_0}^2(\|u_0\|_{H^{\beta_0}}^2 &+A^2) +2 R^2 \Big)\\
        &\lesssim C_1\exp\big( -\frac{NR^2}{C_1}\big),
    \end{split}
    \end{equation}
as well as
\begin{equation*}
    \begin{split}
    \E_{f_0}^N \big[ \|\hat u_N-u_0\|_{N}^2 + \mu_{N,\hat \beta}^2\|\hat u_N\|_{H^{\hat\beta}}^2 \big] &\lesssim \mu_{N,\beta_0}^2(\|u_0\|_{H^{\beta_0}}^2+A^2) +\frac 1N \simeq \mu_{N,\beta_0}^2.
    \end{split}
\end{equation*}
   Using the concentration inequality derived in the previous step with $R=\mu_{N,\beta_0}$ gives that the following event has probability tending to $1$,
   \[ A_N := \Big\{ \|\hat u_N-u_0\|_{N}^2 + \mu_{N,\hat \beta}^2\|\hat u_N\|_{H^{\hat\beta}}^2  \le 2 \mu_{N,\beta_0}^2(\|u_0\|_{H^{\beta_0}}^2 +A^2+1)\Big\},\quad \P_{f_0}^N(A_N)\xrightarrow{N\to\infty} 1. \]
   We may thus restrict the remainder of the proof to this event.
    
    \paragraph*{Step 2: Convergence of the $L^2$-risk} We would now like to show a corresponding rate for the $L^2$-risk. Fix some $\beta \in [\beta_{min}:\beta_{max}]$ and suppose we are on the event $A_N\cap \{\beta = \hat\beta\}$. We have that for some constants $B,B'>0$,
    \[ \|\hat u_N\|_\infty \le B \|\hat u_N\|_{H^{\hat\beta}} \le B'\mu_{N,\beta_0}/\mu_{N,\hat\beta} =: M_N, \]
    We want to apply Lemma \ref{lem-emp-norm} for the function class 
    \[\mathcal U = \{ u-u_0 : \|u\|_{H^{\beta}}\le M_N/B \},\]
    which has entropy bound
    \[ H(\mathcal U,\|\cdot\|_\infty, \rho) \lesssim \big(\frac{M_N}{\rho}\big)^{\frac d{\beta}}. \]

    \textit{i) Case $\beta \ge \beta_0$.} In this case, we have that $M_N\to\infty$, such that the 
    \[ \sup_{v\in\mathcal U}\|v\|_\infty \le 2M_N \]
    for all $N$ large enough. For such $N$, the entropy condition in Lemma \ref{lem-emp-norm} becomes
    \begin{equation*}
        N\delta_N^2 \gtrsim M_N^2 \Big( \frac{M_N}{\delta_N}\Big)^{\frac{d}{\beta}}~~~\iff ~~~\delta_N \gtrsim N^{-\frac{\beta}{2\beta+d}}M_N \simeq \mu_{N,\beta_0}.
    \end{equation*}
    Then we can apply Lemma \ref{lem-emp-norm} with $M=2M_N$ and $\delta_N \simeq \mu_{N,\beta_0}$ to obtain that for constants 
$C,C',C''>0$,
    \begin{equation*}
        \begin{split}
            \P_{f_0}^N\big(A_N \cap \{ &\hat\beta =\beta\},~\|\hat u_N-u_0\|_{L^2} \ge 2\|\hat u_N-u_0\|_N,~\|\hat u_N-u_0 \|_{L^2} \ge C\mu_{N,\beta_0}\big)\\
            &\le \P_{f_0}^N\big(\exists v \in \mathcal U:~\|v\|_{L^2} \ge 2\|v\|_{N},~\|v\|_{L^2}\ge C\mu_{N,\beta_0} \big) \\
            &\lesssim \exp\Big(-\frac{N\mu_{N,\beta_0}^2}{C'M_N^2}\Big)\xrightarrow{N\to\infty} 0.
        \end{split}
    \end{equation*}
    For a sufficiently large constant $L>0$ and since $\P_{f_0}^N(A_N)\to 1$, we obtain
    \[  \P_{f_0}^N\big( \{\hat \beta = \beta \} \cap \{ \|\hat u_N-u_0 \|_{L^2} \ge L\mu_{N,\beta_0}\} \big) \xrightarrow{N\to\infty} 0.\]

    \textit{ii) Case $\beta <\beta_0$.} In this case, we have that $M_N\to 0$, such that the $L^\infty$ bound is of order $1$,
        \[ \sup_{v\in\mathcal U}\|v\|_\infty \le M_N +\|u_0\|_\infty \lesssim 1. \]
Thus, the entropy condition in Lemma \ref{lem-emp-norm} becomes
    \begin{equation*}
        N\delta_N^2 \gtrsim {\delta_N}^{-\frac{d}{\beta}}~~~\iff ~~~\delta_N \gtrsim \mu_{N,\beta}.
    \end{equation*}
    Then we can apply Lemma \ref{lem-emp-norm} with some fixed $M<\infty$ of constant order and $\delta_N \simeq \mu_{N,\beta}$ to obtain that for constants 
$C,C',C''>0$,
    \begin{equation*}
        \begin{split}
            \P_{f_0}^N\big(A_N \cap \{ &\hat\beta =\beta\},~\|\hat u_N-u_0\|_{L^2} \ge 2\|\hat u_N-u_0\|_N,~\|\hat u_N-u_0 \|_{L^2} \ge C\mu_{N,\beta_0} \big)\\
            &\le \P_{f_0}^N\big(\exists v \in \mathcal U:~\|v\|_{L^2} \ge 2\|v\|_{N},~\|v\|_{L^2}\ge C\mu_{N,\beta} \big) \\
            &\xrightarrow{N\to\infty} 0.
        \end{split}
    \end{equation*}
For a sufficiently large constant $L>0$ and since $\P_{f_0}^N(A_N)\to 1$, we obtain
    \[  \P_{f_0}^N\big( \{\hat \beta = \beta \} \cap \{ \|\hat u_N-u_0 \|_{L^2} \ge L\mu_{N,\beta}\} \big) \xrightarrow{N\to\infty} 0.\]
Note that the rate $\mu_{N,\beta}$ is slower than the targeted rate $\mu_{N,\beta_0}$; we will see in the next step that the event $\{\hat\beta <\beta_0\}$ has vanishing probability.

\paragraph*{Step 3: Behaviour of $\hat\beta$} Fix $\beta<\beta_0$ and suppose $\hat \beta =\beta$. On $A_N$, we have that 
\[ \|\hat u_N\|_{L^2}\le  \|\hat u_N\|_{H^{\beta}} \lesssim \mu_{N,\beta_0}/\mu_{N,\beta} \to 0.\]
Moreover, by Step 2 ii), it holds that with probability tending to $1$, $\|\hat u_N - u_0\|_{L^2} \le L \mu_{N,\beta}$. On the event where this is the case,
\[ \|\hat u_N - u_0\|_{L^2} \ge \|u_0\|_{L^2}- \|\hat u_N\|_{L^2} - O(\mu_{N,\beta_0}/\mu_{N,\beta}) \xrightarrow{N\to\infty} \|u_0\|_{L^2} >0, \]
which contradicts $\|\hat u_N - u_0\|_{L^2}= O(\mu_{N,\beta})\to 0$. We thus note that $P_{f_0}^N(\hat \beta=\beta) \to 0$, and hence 
\[ \P_0^N(\hat\beta \ge \beta_0) \xrightarrow{N\to\infty} 1. \]

\paragraph*{Step 4: Inverse problem} Recall the plug-in estimator
\begin{align}
     \hat f_N \in\argmin_{f\in H^{\alpha_{\min}}} J(f),~~~ J(f) = \|L_f \hat u_N- g\|_{L^2}^2 + \nu_{N,\hat \beta}^2 \|f\|_{H^{\alpha_{\min}}}^2,
\end{align}
where $\nu_{N,\beta} = N^{-\frac{\beta-2}{2\beta+d}}$. We work on the event
\[ B_N := A_N\cap \{\hat\beta \ge \beta_0\} \cap \{ \|\hat u_N-u_0\|_{L^2} \le L\mu_{N,\beta_0} \}, \]
for some sufficiently large $L>0$, whose probability tends to $1$ by the preceding proof steps. By the definition of $J$ and using $L_{f_0}u_0=g$, it holds that
\begin{equation*}
    \begin{split}
    J(f_0) &=\|L_{f_0} ( \hat u_N- u_0)\|_{L^2}^2 + \nu_{N,\hat \beta}^2 \|f\|_{H^{\alpha_{\min}}}^2\\
    &\le  \|f_0\|_{C^1}^2\|\hat u_N - u_0\|_{H^2}^2 + \nu_{N,\beta_0}^2 \|f_0\|_{H^{\alpha_{\min}}}^2,
    \end{split}
\end{equation*} 
where we have also used that $\nu_{N,\hat \beta}\le \nu_{N,\beta_0}$. Since $f_0\in H^{\alpha_{min}}$, the second term $ \nu_{N,\beta_0}^2 \|f_0\|_{H^{\alpha_{\min}}}^2= O(\nu_{N,\beta_0}^2)$ is of the desired order.

To estimate the first term, let $u_{0,J}$ denote the high-dimensional wavelet truncation of $u_0$ with $2^{J}\simeq N^{\frac{1}{2\beta_0+d}}$, like in the proof of Theorem \ref{thm-darcy-polytime}. Then
\begin{equation*}
    \begin{split}
        \|\hat u_N - u_0\|_{H^2} \le  \|\hat u_N - u_{0,J}\|_{H^2} +  \|u_0 - u_{0,J}\|_{H^2}=: I+II.
    \end{split}
\end{equation*}
Since $u_0\in H^{\beta_0}(\mathcal O)$, the second term is estimated
\[ II = \|u_0 - u_{0,J}\|_{H^2} \lesssim 2^{J(\beta_0-2)}\|u_0\|_{H^{\beta_0}} \lesssim N^{-\frac{\beta_0-2}{2\beta_0+d}}. \]
We estimate the first term by an interpolation argument. First, we note that 
\[  \|u_{0,J}\|_{H^{\hat \beta}}^2 \simeq \sum_{l\le J,k} S(f)_{lk}^2 2^{2l\hat \beta} \lesssim 2^{2J(\hat\beta-\beta_0)} \|u_0\|_{H^{\beta_0}}^2\lesssim N^{\frac{2(\hat\beta-\beta_0)}{2\beta_0+d}} \le \mu_{N,\beta_0}^2/ \mu_{N,\hat\beta}^2, \]
since $\hat\beta \ge \beta_0$. Then,
\begin{align*}
    \|\hat u_N - u_{0,J}\|_{H^2} &\lesssim \|\hat u_N- u_{0,J}\|_{H^{\hat\beta}}^{2/\hat\beta}  \|\hat u_N - u_{0,J}\|_{L^2}^{(\hat\beta-2)/\hat\beta}\\
    &\lesssim \big(  \|\hat u_N\|_{H^{\hat\beta}} +\| u_{0,J}\|_{H^{\hat\beta}}  \big)^{2/\hat\beta} \big( \|\hat u_N - u_{0}\|_{L^2}+ \|u_0 - u_{0,J}\|_{L^2}\big)^{(\hat\beta-2)/\hat\beta}\\
    &\lesssim \big(\mu_{N,\beta_0}/ \mu_{N,\hat\beta}\big)^{2/\hat\beta}\mu_{N,\beta_0}^{(\hat\beta-2)/\hat\beta}\\
    & = N^{-\frac{\beta_0}{2\beta_0+d}+\frac{2}{2\hat\beta +d }} \le N^{-\frac{\beta_0-2}{2\beta_0+d}}.
\end{align*}
This completes the proof.
\hfill\qed

\subsection{Proof of Theorem \ref{Schroed-rate}}\label{sec-schroed-proof}
The proof is analogous to that of Theorem \ref{Darcy-rate}. Recall
\[
\mu_N := N^{-\frac{\alpha+2}{2(\alpha+2)+d}}
\qquad\text{and}\qquad
\nu_N := N^{-\frac{\alpha}{2(\alpha+2)+d}}.
\]
We begin with the proof for the plug-in M-estimator. By Lemma \ref{lem-schroedinger-boundedness}, there exists $M=M(\alpha,f_{\min}, R)>0$ such that $\|u_0\|_{H^{\alpha+2}}\le M$. We may thus apply Theorem \ref{thm-random-design} with $H=\mathcal U=H^{\alpha+2}(\mathcal O)$, $\eta^*=u_0$ and entropy upper bound 
\[  R+ \int_0^R\Big(\frac{R}{\mu_N\rho}\Big)^{\frac{d}{2(\alpha+2)}}d\rho \lesssim R \mu_N^{-\frac{d}{2(\alpha+2)}} =: \Psi(R). \]
Here, we have used the fact that the localized regression function classes relevant in Theorem \ref{thm-random-design} satisfy $\mathcal U^*(u_0,R)\subseteq \big\{ u:\|u\|_{H^{\alpha+2}}\le R/\mu_N\big\}$, as well as the $L^\infty$ entropy bounds of Sobolev classes from Theorem 4.10.3 in \cite{Triebel1978}, applicable since $\alpha+2>d/2$. The requirement $\sqrt N\delta_N^2 \ge \Psi(\delta_N)$ then yields the convergence rate of
\[\delta_N\gtrsim N^{-\frac{\alpha+2}{2(\alpha+2)+d}}=\mu_N\]
for $\hat u_N$, and we obtain 
\[  \P_{f_0}^N\Big(  \|\hat u_N - \mathcal G(f_0)\|_{N}^2 +\mu_N^2 \|\hat u_N\|^2_{H^{\alpha+2}(\mathcal O)} \ge C \mu_N^2\Big) \lesssim \exp\Big(-\frac{N\mu_N^2}{C}\Big)\xrightarrow{N\to\infty} 0,\]
for some sufficiently large constant $C>0$. Applying Lemma \ref{lem-emp-norm} just as in the proof of Theorem \ref{thm-darcy-polytime} then also shows that 
\[  \P_{f_0}^N\Big(  \|\hat u_N - \mathcal G(f_0)\|_{L^2}^2 +\mu_N^2 \|\hat u_N\|^2_{H^{\alpha+2}(\mathcal O)} \ge C \mu_N^2\Big) \xrightarrow{N\to\infty} 0,\]
for some (possibly larger) constant $C>0$. This proves the first statement of the theorem.

Let $A_N$ denote the event in the previous displayed equation. For the proof of the second statement, we may work throughout on the event $A_N$. On $A_N$, by interpolation, we have that 
\[ \|\hat u_N-u_0\|_{H^2}\lesssim \|\hat u_N-u_0\|_{L^2}^{\frac{\alpha}{\alpha+2}} \|\hat u_N-u_0\|_{H^{\alpha+2}}^{\frac{2}{\alpha+2}} \lesssim N^{-\frac{\alpha}{2(\alpha+2)+d}} \]
Moreover, for any $\beta \in (d/2, \alpha+2)$, the Sobolev embedding theorem and an interpolation argument imply that
\begin{equation*}
    \begin{split}
        \|\hat u_N-u_0\|_\infty &\lesssim \|\hat u_N-u_0\|_{H^\beta}\\
        &\lesssim \|\hat u_N-u_0\|_{L^2}^{\frac{\alpha+2-\beta}{\alpha+2}} \|\hat u_N-u_0\|_{H^{\alpha+2}}^{\frac{\beta}{\alpha+2}}\\
        &\lesssim \mu_N^{\frac{\alpha+2-\beta}{\alpha+2}} \xrightarrow{N\to\infty} 0.
    \end{split}
\end{equation*}
Let $u_{\min}:= \inf_{x\in\mathcal O} u_0(x)$. Under our standing assumptions $\alpha>d/2$, $g_{\min}>0$ and $g\in C^\infty(\partial\O)$ with $\inf_{\partial\O}g\ge g_{\min}$, a Feynman--Kac argument implies a uniform positivity bound $\inf_{x\in\O}u_f(x)\ge g_{\min}e^{-c\|f\|_\infty}>0$ for all $f\in\mathcal F$, see display (5.36) in \cite{NVW18}. In particular, it holds that $u_{\min}>0$. Thus \[\P_{f_0}^N(\inf_{x\in\O} \hat u_N \ge u_{\min} /2) \xrightarrow{N\to\infty} 1.\]
By the definition of $\hat f_N$, it holds that 
\begin{equation*}
    \begin{split}
        \|L_{\hat f_N}\hat u_N\|_{L^2}^2 + \nu_N^2\|\hat f_N\|_{H^\alpha}^2 &\le  \|L_{f_0}\hat u_N\|_{L^2}^2 +\nu_N^2\|f_0\|_{H^\alpha}^2 \\
        &\le \|L_{f_0}\hat u_N-L_{f_0}u_0\|_{L^2}^2 +\nu_N^2\|f_0\|_{H^\alpha}^2 \\
        &\lesssim \|\hat u_N-u_0\|_{H^2}^2 +\nu_N^2\|f_0\|_{H^\alpha}^2 \\
        &\lesssim \nu_N^2,
    \end{split}
\end{equation*}
which in particular also yields that for $\beta\in (0,\alpha-d/2)$, $\|\hat f_N\|_{C^\beta} \lesssim \|\hat f_N\|_{H^\alpha}\le C''$ for some constant $C''$, on the event $A_N$. Thus, an application of the generalized stability estimate from Lemma \ref{schroed-generalized-stability} with the first triplet consisting of $(u_0,f_0,0)$ and the second triplet consisting of $(\hat u_N,\hat f_N, L_{\hat f_N}\hat u_N)$ yields that on the event $A_N$, 
\[ \|\hat f_N-f_0\|_{L^2} \lesssim \|\hat u_N-u_0\|_{H^2} + \|L_{\hat f_N}\hat u_N\|_{L^2}\lesssim \nu_N, \]
which proves the second claim.
\hfill \qed

\section{Proof of Theorem \ref{thm-random-design}}\label{appendix-M-estimationproofs} 

The proof works by adapting ideas from \cite{V00} to the current setting with `augmented parameters' $(\eta,\theta)$. Fix any $(\theta^*,\eta^*)$ and let us write 
\[\hat u := u_{\hat \eta},~~~~~u^*:= u_{\eta^*}.\] By the definition of $(\hat \eta_N,\hat\theta_N)$, it holds that
\[ \frac 1N \sum_{i=1}^N (Y_i-u_{\hat \eta}(X_i))^2 + \mathcal R(\hat \eta,\hat \theta) \le \frac 1N \sum_{i=1}^N (Y_i-u^*(X_i))^2 + \mathcal R(\eta^*,\theta^*),
 \]
which is the same as
    \begin{equation}\label{rd:design:base}
    \begin{split}
        \tau^2((\hat \eta,\hat\theta),u_0) - \tau^2((\eta^*,\theta^*),u_0)  \le \frac 2N \sum_{i=1}^N\eps_i (\hat u(X_i)-u^*(X_i)).
    \end{split}
    \end{equation}
Next, we note that the functional $\tau^2$ satisfies the following `triangle-type' inequality in the $\eta$-variable: For any $\eta,\eta', \eta''\in H$ with corresponding regression functions $u,u',u''$ and any $\theta, \theta'\in\Theta$, using the inequality $\|a-b\|_N^2\le 2\|a-c\|_N^2+2\|c-b\|_N^2$
for the empirical norm and the non-negativity of the penalty
$\mathcal R\ge 0$, we have
\begin{align*}
    \tau^2((\eta,\theta),u_0)&\le 2\tau^2((\eta,\theta),u')+2\tau^2((\eta',\theta'),u_0),\\
    \tau^2((\eta,\theta),u'')&\le 2\tau^2((\eta,\theta),u')+2\tau^2((\eta'',\theta'),u'),
\end{align*}
which implies
	\begin{align}
			\tau^2((\hat \eta, \hat \theta),u^*)&\ge \frac 12 \tau^2((\hat \eta, \hat \theta),u_0)-\tau^2((\eta^*,\theta^*),u_0),      \label{rd:design:tau1}\\
			\tau^2((\hat \eta, \hat \theta),u_0)&\ge \frac 12 \tau^2((\hat \eta, \hat \theta),u^*)-\tau^2((\eta^*,\theta^*),u_0).    \label{rd:design:tau2}
	\end{align}
Denote the event whose probability we want to bound by
\begin{equation}\label{rd:design:base2}
    A_N := \big\{\tau^2((\hat \eta,\hat \theta),u_0)\ge 2\big(\tau^2((\eta^*, \theta^*),u_0)+R^2) \big\}.
\end{equation}
On $A_N$, using (\ref{rd:design:tau1}), it holds that
\[ \tau^2((\hat \eta, \hat \theta),u^*)\ge R^2. \]
Moreover, by the definition of $A_N$ and using \eqref{rd:design:tau2}, we obtain
\begin{align*}
    \tau^2((\hat \eta, \hat \theta),u_0) -\tau^2((\eta^*,\theta^*),u_0) &=\frac 13 \tau^2((\hat \eta, \hat \theta),u_0) + \frac 23 \tau^2((\hat \eta, \hat \theta),u_0) -\tau^2((\eta^*,\theta^*),u_0) \\
    &\ge \frac 13 \tau^2((\hat \eta, \hat \theta),u_0) + \frac 13 \tau^2((\eta^*,\theta^*),u_0) \\
    &\ge \frac 16 \tau^2((\hat \eta,\hat\theta),u^*).
\end{align*}
In light of (\ref{rd:design:base}), this further implies
\[ 
\frac 1N \sum_{i=1}^N\eps_i (\hat u(X_i)-u^*(X_i)) \ge \frac 1{12} \tau^2((\hat \eta,\hat\theta),u^*). \]
The preceding argument allow us to employ `slicing' to upper bound the probability of $A_N$,
\begin{align*}
    \P_0^N(A_N) &\le \sum_{s=1}^\infty ~\P_0^N \big(A_N ~\text{and}~\tau^2((\hat \eta, \hat\theta),u^*)\in \big[2^{2(s-1)}R^2, 2^{2s}R^2 \big] \big)\\
    &\le \sum_{s=1}^\infty \P_0^N \Big(\sup_{u\in \mathcal U(\eta^*,2^{s}R)} \Big|\frac 1{\sqrt N} \sum_{i=1}^N\eps_i (u(X_i)-u^*(X_i)) \Big| \ge \frac 1{48} \sqrt N 2^{2s} R^2 \Big)\\
    &:= \sum_{s=1}^\infty P_s
\end{align*}
We write $\mathcal U_s = \mathcal U(\eta^*,2^{s}R)$. For each $s\ge 1$, conditional on the design points $(X_i)_{i=1}^N$, the process 
\[X_u =\frac 1{\sqrt N} \sum_{i=1}^N\eps_i (u(X_i)-u^*(X_i)),~~~~ u \in \mathcal U_s, \] is centered Gaussian with covariance metric 
\[d^2(u,u') = \frac 1N \sum_{i=1}^N \sigma_i^2 (u(X_i)-u'(X_i))^2 \le  \sigma^2 \| u-u' \|_{N}^2,\] with respect to which the set $\mathcal U_s$ has diameter
\[D_s:= \sup_{u,u'\in \mathcal U_s} d(u,u') \le 2\sigma  2^{s}R. \]
Dudley's inequality for suprema of sub-Gaussian processes (see, e.g.,~Theorems 2.3.6 and 2.3.7 in \cite{GN16} or Theorem 2.2.4 in \cite{VW96}) gives that for any $v\in \mathcal U_s$
\begin{align*}
    \E_0^N \big[ \sup_{u\in \mathcal U_s} \big| X_u\big| \big] &\le \E_0^N |X_{v}| +  4\sqrt 2 \int_{0}^{D_s/2} \sqrt {H(\mathcal U_s,\sigma\|\cdot\|_N,\rho)} d\rho \\
    &\le \sigma 2^s R + 4\sqrt 2 \int_0^{\sigma 2^s R}\sqrt {H(\mathcal U_s,\sigma\|\cdot\|_N,\rho)} d\rho\\
    &\le 4\sqrt 2 \Psi(2^sR) \le \frac{1}{96} \sqrt N 2^{2s}R^2,
\end{align*}
where we can ensure the last inequality by choosing the constant $C>0$ in the hypothesis \eqref{entropy-ineq} large enough. Using this and the Borell-Sudakov-Tsirelson inequality (see Theorem 2.5.8 in \cite{GN16}), noting that the one-dimensional variances are upper bounded by
$\sup_{u\in \mathcal U_s}\text{Var}(X_u)\le \sigma^2 2^{2s}R^2$,
\begin{align*}
    P_s &\le \P_0^N\Big(\sup_{u\in \mathcal U_s} \big| X_u\big|- \E_0^N\sup_{u\in \mathcal U_s} \big|X_u\big| \ge \frac{1}{96} \sqrt N 2^{2s}R^2 \Big) \le \exp \Big( -\frac{N2^{2s}R^2}{2\cdot 96^2\sigma^2}\Big).
\end{align*}
Summing the previous inequalities gives 
\begin{equation*}
    P_0^N(A_N)\le \sum_{s=1}^\infty P_s\le C\exp\big( - NR^2/C\sigma^2\Big),
\end{equation*}
where we also used that under the hypothesis (\ref{entropy-ineq}), it holds that $\sqrt N \delta_N^2\ge C\sigma \delta_N$ such that $NR^2/(2*96^2\sigma^2)$ can be bounded away from zero by some universal constant for any $R\ge \delta_N$. This proves \eqref{conc-random-design}.

To show \eqref{MSE-random-design} for $k=1$, we integrate the concentration inequality \eqref{conc-random-design}, as in Lemma 2.2 of \cite{V01}. Let us write $\hat \tau^2 = \tau^2((\hat \eta,\hat\theta),u_0)$. For $\delta_N$ satisfying \eqref{entropy-ineq}, let
\[t_N^*=2(\tau^2((\eta^*,\theta^*),u_0)+\delta_N^2).\]
Then
\begin{equation*}
    \begin{split}
        \E_0^N[\hat \tau^2] &\le t_N^* + \int_{t_N^*}^\infty \P_0^N(\hat\tau^2 \ge t) dt \le t_N^*+ C\int_{\delta_N}^\infty \exp\Big(-\frac{Nu^2}{C\sigma^2}\Big) du^2\\
        &\le t_N^* + \frac{C'\sigma^2}{N} \exp\Big(-\frac{N\delta_N^2}{C\sigma^2}\Big)\le t_N^* + \frac{C'\sigma^2}{N},
    \end{split}
\end{equation*}
which proves the claim.

We now show the claim for higher moments of $\hat\tau$. For integer $k\ge 2$, we have
\begin{equation}\label{2k-moment}
    \begin{split}
        \E_0^N[\hat \tau^{2k}] &\le (t_N^*)^k + \int_{(t_N^*)^k}^\infty \P_0^N(\hat\tau^2 \ge t^{1/k}) dt \le (t_N^*)^k+ C\int_{\delta_N^{2k}}^\infty \exp\Big(-\frac{Nt^{1/k}}{C\sigma^2}\Big) dt
    \end{split}
\end{equation}
Since $\Psi(R)\ge \sigma R$, the assumption $\sqrt{N}\delta_N^2 \ge C\Psi(\delta_N)$
implies $N\delta_N^2/(C\sigma^2)\ge 1$, so that Lemma~\ref{lem:k-moment-tail-bound}
applies with $\gamma=1/k$, $a=\delta_N^{2k}$ and $L=N/(C\sigma^2)$. Using that lemma, we obtain that
\begin{equation*}
    \begin{split}
    \int_{\delta_N^{2k}}^\infty \exp\Big(-\frac{Nt^{1/k}}{C\sigma^2}\Big) dt \lesssim \frac{\sigma^2\delta_N^{2k-2}}{N} \exp\Big(-\frac{N\delta_N^2}{C\sigma^2}\Big) \le \frac{\sigma^2\delta_N^{2k-2}}{N},
    \end{split}
\end{equation*}
as desired. \hfill \qed

\section{Two auxiliary lemmas}

\subsection{Concentration of the empirical norm}
\begin{lem}[Concentration of empirical norm]\label{lem-emp-norm}
Let $\mathcal O\subseteq \R^d$ be some open subset with finite Lebesgue measure, and let $X_i \sim^{i.i.d.}\text{Unif}(\mathcal O)$. Let $\mathcal U$ be a class of real-valued functions on $\mathcal O$ with uniform upper bound 
\[ \sup_{u\in\mathcal U} \|u\|_{L^\infty(\mathcal O)} =: M<\infty, \]
where $\|u\|_{L^\infty(\mathcal O)} = \sup_{x\in \mathcal O}|u(x)|$. There exists a universal constant $C>0$ such that for any $(\delta_N:N\ge 1)$ satisfying
\begin{equation}\label{emp-norm-condition}
    N\delta_N^2 \ge C M^2 H(\mathcal U,\|\cdot\|_\infty,\delta_N),
\end{equation}
and any $R\ge C(\delta_N + M/\sqrt N)$, it holds that
\[
\mathbb{P}\Big( \| u \|_{L^2(\mathcal O)} \ge 2 \|u\|_N ~~\text{for some}~~u\in\mathcal U~\text{with}~\|u\|_{L^2(\mathcal O)}\ge R\Big)
\le 2 \exp\left( -\frac{NR^2}{C M^2} \right).
\]
\end{lem}
\begin{proof}
    Lemma~\ref{lem-emp-norm} is a direct specialization of Lemma~A.5 in \cite{reinhardt2024statistical} to the case of uniform design $X_i\sim\mathrm{Unif}(\O)$ and the function class $\mathcal U$ with uniform envelope $M=\sup_{u\in\mathcal U}\|u\|_\infty$.
\end{proof}

For classes $\mathcal U$ which belong to a sufficiently regular Sobolev ball
\[\mathcal U\subseteq  \{u\in H^\alpha(\mathcal O):\|u\|_{H^\alpha}\le M\},~~~\alpha>d/2,\]
the preceding lemma implies the following. By virtue of the Sobolev embedding $\|\cdot\|_\infty\lesssim \|\cdot\|_{H^\alpha}$ and using the entropy bound $H(\mathcal U,\|\cdot\|_\infty,\delta_N)\lesssim (M/\delta_N)^{d/\alpha}$ (e.g. \cite[Thm.~4.10.3]{Triebel1978}), we see that the condition \eqref{emp-norm-condition} is fulfilled for 
\[\delta_N \gtrsim MN^{-\frac{\alpha}{2\alpha+d}}, \]
where $\gtrsim$ denotes a constant which is independent of $M,N$. It follows that for some constant $C'>0$, 
\[
\mathbb{P}\Big( \| u \|_{L^2(\mathcal O)} \ge 2 \|u\|_N ~~\text{for some}~~\|u\|_{H^\alpha} \le M,~~\|u\|_{L^2(\mathcal O)}\ge R \Big)
\le 2 \exp\left( -\frac{NR^2}{C' M^2} \right),
\]
with some constant $C'$.

\subsection{A basic tail inequality}

The following tail bound is used repeatedly in our proofs.
\begin{lem}\label{lem:k-moment-tail-bound}
For any $\gamma\in(0,1)$, there exists a constant $C_\gamma>0$ such that for all
$L,a>0$ with $La^\gamma\ge 1$,
\[
\int_a^\infty e^{-Lt^\gamma}\,dt
\le
C_\gamma \frac{a^{1-\gamma}}{L}\,e^{-La^\gamma}.
\]
\end{lem}

 \begin{proof}
 One sees this via substituting \(u = L t^\gamma\), so that
\[
t = \left(\frac{u}{L}\right)^{1/\gamma},
\qquad
dt = \frac{1}{\gamma} L^{-1/\gamma} u^{\frac{1}{\gamma}-1}\,du.
\]
Then
\[
\int_a^\infty e^{-L t^\gamma}\,dt
=
\frac{1}{\gamma} L^{-1/\gamma}
\int_{L a^\gamma}^\infty
u^{\frac{1}{\gamma}-1} e^{-u}\,du
=
\frac{1}{\gamma} L^{-1/\gamma}
\Gamma\!\left(\frac{1}{\gamma},\, L a^\gamma\right),
\]
where $\Gamma$ is the upper incomplete gamma function,
\[
\Gamma(s,z) = \int_z^\infty u^{s-1} e^{-u}\,du.
\]
For fixed \(s > 0\), the tail of $\Gamma$ satisfies
\[
\Gamma(s,z) \le C_s\, z^{s-1} e^{-z}
\qquad \text{for } z \ge 1.
\]
Applying this with \(s = \frac{1}{\gamma}\) and \(z = L a^\gamma\), one obtains 
\[
\int_a^\infty e^{-L t^\gamma}\,dt
\le
C_\gamma\, \frac{a^{1-\gamma}}{L}\, e^{-L a^\gamma},
\]
with \(C_\gamma > 0\) depending only on \(\gamma\).
 \end{proof}

\newpage

\end{document}